\documentclass[11pt,letterpaper]{article}
\usepackage[margin=1in]{caption}
\usepackage{geometry}
\usepackage{amsmath,amssymb,color, mathtools, tikz, tikz-cd, multicol,subfig}
\usepackage{amscd}
\usepackage{amsthm}
\numberwithin{equation}{section}

\newtheorem{theorem}{Theorem}[section]
\newtheorem{proposition}[theorem]{Proposition}

\newtheorem{question}[theorem]{Question}
\newtheorem{corollary}[theorem]{Corollary}

\newtheorem{lemma}[theorem]{Lemma}

\theoremstyle{definition}
\newtheorem{definition}[theorem]{Definition}
\newtheorem{example}[theorem]{Example}
\newtheorem{construction}[theorem]{Construction}
\newtheorem{remark}[theorem]{Remark}

\newcommand{\merge}{\scalebox{1.2}{$\hspace{0.1em}\triangleright$}}
\newcommand{\conv}{\mathrm{conv}}
\newcommand{\aff}{\mathrm{aff}}
\newcommand{\dist}{\mathrm{dist}}

\newcommand{\R}{\mathbb{R}}
\newcommand{\lattice}{\mathcal{L}}
\newcommand{\CP}{\mathrm{CP}}

\newcommand{\inter}{\mathrm{int}}

\title{The merging operation and $(d-i)$-simplicial $i$-simple $d$-polytopes}
\author{
	Isabella Novik\thanks{Research of IN is partially\textsl{} supported by NSF grant DMS-1953815 and by Robert R.~\&  Elaine F.~Phelps Professorship in Mathematics. }\\
	\small Department of Mathematics\\[-0.8ex]
	\small University of Washington\\[-0.8ex]
	\small Seattle, WA 98195-4350, USA\\[-0.8ex]
	\small \texttt{novik@uw.edu}
	\and 
	Hailun Zheng\\
	\small Department of Mathematics \& Statistics\\[-0.8ex]
	\small University of Houston-Downtown\\[-0.8ex]
	\small One Main Street, Houston, TX 77002, USA \\[-0.8ex]
	\small \texttt{zhengh@uhd.edu}
}
\begin{document}
\maketitle
\begin{center}{\em Dedicated to G\"unter M.~Ziegler on the occasion of his $60$th birthday.}\end{center}

\begin{abstract}
We define a certain merging operation that given two $d$-polytopes $P$ and $Q$ such that $P$ has a simplex facet $F$ and $Q$ has a simple vertex $v$ produces a new $d$-polytope $P\merge Q$ with $f_0(P)+f_0(Q)-(d+1)$ vertices. We show that if for some $1\leq i\leq d-1$, $P$ and $Q$ are $(d-i)$-simplicial $i$-simple $d$-polytopes,  then  so is $P\merge Q$. We then use this operation to construct new families of $(d-i)$-simplicial $i$-simple $d$-polytopes. Specifically, we prove that for all $2\leq i \leq d-2\leq 6$ with the exception of $(i,d)=(3,8)$ and $(5,8)$, there is an infinite family of $(d-i)$-simplicial  $i$-simple $d$-polytopes; furthermore,  for all $2\leq i\leq 4$, there is an infinite family of self-dual $i$-simplicial $i$-simple $2i$-polytopes. Finally, we show that for any $d\geq 4$, there are $2^{\Omega(N)}$ combinatorial types of  $(d-2)$-simplicial $2$-simple $d$-polytopes with at most $N$ vertices.
\end{abstract}

	\section{Introduction}
	 This paper is devoted to the fascinating class of $(d-i)$-simplicial $i$-simple $d$-polytopes.
	 
	 A polytope is the convex hull of finitely many points in $\R^d$. The $f$-vector of a polytope encodes the number of faces of each dimension. For brevity, we refer to $d$-dimensional polytopes as $d$-polytopes. A $d$-polytope $P$ is called simplicial if every facet of $P$ contains exactly $d$ vertices. Similarly, a $d$-polytope $P$ is simple, if every vertex of $P$ is in exactly $d$ facets. A polytope that is neither simplicial nor simple is called a general polytope.

	 While polytopes have been studied since antiquity, it is not an exaggeration to say that we still know very little about them.  To be more precise, various combinatorial invariants such as the face lattice and the $f$-vectors of $3$-polytopes are well understood thanks to the Steinitz theorem \cite{Steinitz}. In higher dimensions, the celebrated $g$-theorem provides a complete characterization of the $f$-vectors of simplicial and simple polytopes \cite{BilleraLee,Stanley80}. Yet, at present, the problem of characterizing the $f$-vectors of general $d$-polytopes for $d\geq 4$ is completely out of reach already in dimension four; see, for instance, \cite{Bayer87, BrinkZieg, PafWer,Zieg02,Zieg04}.

One may wonder if studying the $f$-vectors of general polytopes that share some properties with both simplicial and simple polytopes might be a more approachable task. To this end, a $d$-polytope is called {\em $(d-i)$-simplicial} if all of its $(d-i)$-faces are simplices, and it is {\em $i$-simple} if every $(d-i-1)$-face is contained in exactly $i+1$ facets. In particular,  the class of $(d-1)$-simplicial $1$-simple $d$-polytopes coincides with the class of simplicial $d$-polytopes, while the class of $1$-simplicial $(d-1)$-simple $d$-polytopes is the class of simple $d$-polytopes. In other words, as $i$ varies between $1$ and $d-1$, $(d-i)$-simplicial $i$-simple $d$-polytopes interpolate between the class of simplicial $d$-polytopes and the class of simple $d$-polytopes. It is also worth mentioning that the $f$-vector of a $2$-simplicial $2$-simple $4$-polytope is symmetric.

Do $(d-i)$-simplicial $i$-simple $d$-polytopes  exist when $2\leq i\leq d-2$? (Cf.~\cite[Problem 19.5.23]{Kalai-skeletons}.) While various conjectures (see, for instance \cite[Exercise 9.7.7(iii)]{Gru-book}) suggest that there should be many such polytopes, not much is known. The first infinite family of $2$-simplicial $2$-simple $4$-polytopes was constructed by Eppstein, Kuperberg and Ziegler \cite{EKZ}. Their approach was generalized by Paffenholz and Ziegler \cite{PafZieg} who established the existence of infinite families of $(d-2)$-simplicial $2$-simple $d$-polytopes for all $d\geq 4$. Notably, the minimum number of vertices in their $d$-dimensional construction is $2(d+1)$, realized by $\conv(\Sigma\cup \Sigma^*)$, where $\Sigma$ is a $d$-simplex whose $(d-3)$-faces are tangent to the unit sphere $\mathbb{S}^{d-1}$. Additional infinite families of $2$-simplicial $2$-simple $4$-polytopes were constructed by Paffenholz and Werner \cite{PafWer}: all their polytopes are elementary (i.e., have $g^{\mathrm{toric}}_2=0$) and have at least one simplex facet.

As for larger values of $i$, the $d$-dimensional demicube with $d\geq 4$ (also known as the half-cube) is $3$-simplicial $(d-3)$-simple while its dual is $(d-3)$-simplicial $3$-simple (see \cite[Exercise 4.8.18]{Gru-book}). Furthermore, the Gosset--Elte polytopes that arise from Wythoff's construction provide finitely many examples of $(d-i)$-simplicial $i$-simple $d$-polytopes for $d\leq 8$ and $2\leq i\leq d-2$ \cite{Coexter}. These are essentially all known to-date examples of $(d-i)$-simplicial $i$-simple $d$-polytopes with  $2\leq i\leq d-2$. In particular, it is not known whether a $5$-simplicial $5$-simple $10$-polytope exists. In light of this, we further pose the following questions.

\begin{question}\label{Question 1} \quad
	\begin{enumerate}
		\item Let $d\geq 4$. What is the minimum number of vertices that a non-simplex $(d-2)$-simplicial $2$-simple $d$-polytope can have?
		\item Let $d\geq 6$ and let $3\leq i\leq d/2$. Are there infinite families of $(d-i)$-simplicial $i$-simple $d$-polytopes? What is the minimum number of vertices that such a non-simplex polytope can have?
	\end{enumerate}
\end{question}

The goal of this paper is to provide new infinite families of $(d-i)$-simplicial $i$-simple $d$-polytopes for some values of $i$ and $d$. To achieve this, we define a certain merging operation that given two $d$-polytopes $P$ and $Q$, where $P$ has a simplex facet and $Q$ has a simple vertex, outputs a new $d$-polytope. This operation is modeled on a familiar notion of connected sums of simplicial polytopes, but designed in a way that preserves the property of being $(d-i)$-simplicial $i$-simple. Using this operation, we establish the following results:

\begin{enumerate}
	\item   There exist infinite families of $(d-i)$-simplicial $i$-simple d-polytopes for all pairs $(i,d)$ such that $2\leq i\leq d-2\leq 6$ and $(i,d)$ is not $(3,8)$ or $(5,8)$; see Theorem \ref{thm: lower dimensional infinite families}. This partially answers Question \ref{Question 1}(2) and \cite[Problem 19.5.23]{Kalai-skeletons}.
	
	\item There exist infinite families of self-dual $i$-simplicial $i$-simple $2i$-polytopes for $2 \leq i\leq 4$; see Theorem \ref{thm:inf-many-self-dual}. This partially answers \cite[Problem 19.5.24]{Kalai-skeletons}.
	
	\item For all $d\geq 4$, there are $2^{\Omega(N)}$ combinatorial types of $(d-2)$-simplicial $2$-simple $d$-polytopes with at most $N$ vertices; see Theorem \ref{prop: counting combinatorial types}.
\end{enumerate}

To prove the last result, we construct a higher-dimensional analog of the unique $2$-simplicial $2$-simple $4$-polytope with nine vertices. (This $4$-polytope is called $P_9$ in \cite{PafWer}; it has the minimum number of vertices among all non-simplex $2$-simplicial $2$-simple $4$-polytopes.) We then apply the merging operation to produce new infinite families of $(d-2)$-simplicial $2$-simple $d$-polytopes.

As for the second result, several examples of (non-simplex) self-dual $2$-simplicial $2$-simple $4$-polytopes were known before, among them polytopes $P_9$ and $P_{10}$ from \cite{PafWer}. In fact, \cite{Paff} provides a (different) infinite family of self-dual $2$-simplicial $2$-simple $4$-polytopes, that for instance, includes the $24$-cell.  An interesting infinite family of self-dual $d$-polytopes that are neither $j$-simplicial nor $i$-simple (for any $d\geq3$ and $j,i>1$) is the family of multiplexes constructed by Bisztriczky \cite{Bisztr}.

The outline of the paper is as follows. We review several definitions related to polytopes and face lattices in Section 2. Section 3 serves as a warm-up section where we discuss the minimum number of vertices that a non-simplex $3$-simplicial $2$-simple $5$-polytope can have. In Section 4, we introduce and study the merging operation that applies to pairs of polytopes one of which has a simplex facet and another a simple vertex. This operation has several interesting properties;  see, for instance, Theorem \ref{main theorem} and Theorem \ref{prop: face lattice of the merge}. Sections 5 and 6 form the most crucial part of this paper: there, we utilize the merging operation and its properties to provide our promised constructions of new $(d-i)$-simplicial $i$-simple $d$-polytopes. Specifically, in Section 5.1, we construct infinite families of $(d-i)$-simplicial $i$-simple $d$-polytopes for $d\leq 8$. In Section 5.2, we construct infinite families of self-dual $i$-simplicial $i$-simple $2i$-polytopes for $i\leq 4$. In Section 6.1, we revisit the $2$-simplicial $2$-simple $4$-polytopes providing several new constructions. Finally, in Section 6.2, we produce a higher-dimensional analog of $P_9$ and use it to construct exponentially many (in $N$) combinatorial types of $(d-2)$-simplicial $2$-simple $d$-polytopes with at most $N$ vertices.

	\section{Preliminaries}
	A {\em polytope} $P \subseteq \R^d$ is the convex hull of a finite set of points in $\R^d$. The {\em dimension} of $P$ is the dimension of the affine span of $P$. For brevity, we say that $P$ is a {\em $d$-polytope} if $P$ is $d$-dimensional. In what follows, we always assume that $P \subseteq \R^d$ is a $d$-polytope. 
	
	A hyperplane $H\subseteq \R^d$ is a {\em supporting hyperplane} of $P$ if $P$ is contained in one of the two closed half-spaces determined by $H$. A {\em  (proper) face of $P$} is the intersection of $P$ with any supporting hyperplane of $P$. A face of a polytope is by itself a polytope. We refer to $(d-1)$-faces of $P$ as {\em facets} of $P$, to $(d-2)$-faces as {\em ridges}, to $1$-faces as {\em edges}, and to $0$-faces as {\em vertices}. We denote by $V(P)$ the vertex set of $P$. If $V(P)$ consists of $d+1$ affinely independent points, then $P$ is a {\em $d$-simplex}; we denote it by $\sigma_d$. 
    
    The face poset of $P$, $\lattice(P)$, is the set of faces of $P$ (including $P$ and $\emptyset$) ordered by inclusion. The face poset of $P$ is a lattice. We usually write the maximum element of $\lattice(P)$ (namely, $P$) as $\hat{1}$ and the minimum element (namely, $\emptyset$) as $\hat{0}$. For a subset $S$ of $\lattice(P)$, we let $\vee S$ and $\wedge S$ denote the join and the meet of elements of $S$, respectively. 
		
		By using translation, if necessary, we can always assume that the origin lies in the interior of  $P$. The dual polytope $P^*$ of $P$ is then defined as  $$P^*=\{y\in \R^d: \;y^tx\leq 1, \;\forall x\in P\}.$$
	For every $P$, there is an {\em order-reversing} bijective map $\phi: \lattice(P)\to \lattice (P^*)$; by slight abuse of notation, we also denote by $\phi$ the (inverse) map $\lattice(P^*)\to \lattice (P^{**})=\lattice (P)$. We say that $P$ is {\em self-dual} if $\lattice(P)$ and $\lattice(P^*)$ are isomorphic (i.e., there is an {\em order-preserving} bijective map).
	
	Let $1\leq i\leq d-1$. A $d$-polytope $P$ is {\em $i$-simplicial} if all of its $i$-faces are simplices; equivalently, if all of its $i$-faces have $i+1$ vertices. Similarly, $P$ is {\em $i$-simple} if every $(d-i-1)$-face is contained in exactly $i+1$ facets. The class of $(d-1)$-simplicial $d$-polytopes is known as the class of simplicial $d$-polytopes, while the class of $(d-1)$-simple $d$-polytopes is known as the class of simple $d$-polytopes. In particular, if $P$ is $i$-simplicial, then the interval $[\hat{0}, \tau]$ is a Boolean lattice for any face $\tau$ with $\dim \tau \leq i$. Likewise, if $P$ is $i$-simple, then $[\tau, \hat{1}]$ is Boolean for any face $\tau$ with $\dim \tau\geq d-i-1$. Hence $P$ is $i$-simplicial if and only if $P^*$ is $(d-i)$-simple. 
	
	If $v$ is a vertex of $P$, then the {\em vertex figure of $P$ at $v$}, denoted $P/v$, is the polytope obtained by intersecting $P$ with a hyperplane $H$ that has $v$ on one side and all other vertices of $P$ on the other side. The combinatorial type of $P/v$ does not depend on the choice of $H$. In fact, $\lattice(P/v)$ is exactly the interval $[v, \hat{1}]$ in $\lattice(P)$. We say that a vertex $v$ of a $d$-polytope $P$ is simple if $P/v$ is a simplex, or equivalently, if $v$ belongs to exactly $d$ facets of $P$.
	
	If $P$ is a simplicial polytope, then the collection of vertex sets of faces of $P$, including $\emptyset$ but not including $P$ itself, forms an {\em abstract simplicial complex} $\partial P$ called the {\em boundary complex} of $P$. When $V$ is a finite set, we let $\partial \overline{V}:=\{\tau\subset V : \tau\neq V\}$ denote the boundary complex of an abstract  simplex with vertex set $V$. 
	
	Consider a $d$-polytope $P\subset \R^d\times \{0\}$ and a $d'$-polytope $Q\subset \{0\}\times \R^{d'}$ such that the origin is in the relative interior of both $P$ and $Q$. The polytope $P\oplus Q:=\conv(P\cup Q)$ is called the {\em free sum} of $P$ and $Q$. All faces of $P\oplus Q$ are of the form $\conv(F\cup G)$, where $F\neq P$ is a face of $P$ and $G\neq Q$ is a face of $Q$. Consequently, if $P$ and $Q$ are simplicial polytopes then the boundary complex of $P\oplus Q$ coincides with the {\em join} of $\partial P$ and $\partial Q$: 
	$$\partial (P\oplus Q)=\partial P*\partial Q:=\{\sigma\cup \tau: \sigma\in \partial P, \tau\in\partial Q\}.$$
    
	For a $d$-polytope $P$, we let $f(P)=(f_0(P), f_1(P),\ldots,f_{d-1}(P))$ be the {\em $f$-vector} of $P$; here $f_i(P)$ denotes the number of $i$-faces of $P$. Also, for $0\leq i<j\leq d-1$, we let $f_{i,j}(P)$ denote the number of pairs of faces $F_i\subset F_j$ of $P$ such that $\dim F_i=i$ and $\dim F_j=j$.
		
		To conclude this section, we note that for all $0\leq i\leq d-1$, $f_i(P)=f_{d-i-1}(P^*)$. This is immediate from the existence of an order-reversing bijection $\phi: \lattice(P)\to \lattice (P^*)$.

	\section{A warm-up: the minimum number of vertices} 
	As mentioned in the introduction, for every $d\geq 4$, there exists a $(d-2)$-simplicial $2$-simple $d$-polytope with $2(d+1)$ vertices. Furthermore, for $d=4$, there is a $2$-simplicial $2$-simple $4$-polytope with only $9$ vertices. Are there $(d-2)$-simplicial $2$-simple $d$-polytopes with fewer than $2d+2$ vertices for $d>4$? (Cf.~Question \ref{Question 1}(1).) The goal of this warm-up section is to answer this question for $d=5$; see Proposition \ref{vertex number:d=5,i=2}. To do this, we first establish a criterion  that the $f$-vectors of $(d-i)$-simplicial $i$-simple $d$-polytopes (if they exist) must satisfy; cf.~\cite[Exercise 9.7.7(ii)]{Gru-book}. We include the proof for completeness.
	\begin{lemma}\label{lm: f-vector criterion}
		 Let $d\geq 2$ and $1\leq i\leq d-1$. Let $P$ be a $(d-i)$-simplicial $d$-polytope. Then $P$ is $i$-simple if and only if $(d-i+1)f_{d-i}(P)=(i+1)f_{d-i-1}(P).$
	\end{lemma}
	\begin{proof}
		If $P$ is $(d-i)$-simplicial, then every $(d-i)$-face of $P$ is a simplex; hence, every $(d-i)$-face contains $d-i+1$ faces of dimension $d-i-1$. This means that $f_{d-i-1,d-i}(P)=(d-i+1)f_{d-i}(P)$. On the other hand, a $(d-i-1)$-face of any $d$-polytope is contained in at least $i+1$ faces of dimension $d-i$. Thus, $f_{d-i-1,d-i}(P)\geq (i+1)f_{d-i-1}(P)$, and we conclude that $(d-i+1)f_{d-i}(P)=f_{d-i-1,d-i}(P) \geq (i+1)f_{d-i-1}(P)$. Furthermore, equality holds if and only if every $(d-i-1)$-face is in exactly $i+1$ faces of dimension $d-i$ which happens if and only if $P$ is $i$-simple.	
	\end{proof}
	
	\begin{corollary} \label{cor:equality}
		For all $i\geq 1$, an $i$-simplicial $2i$-polytope $P$ is $i$-simple if and only if $f_{i-1}(P)=f_i(P)$.
	\end{corollary}

	\begin{proposition} \label{vertex number:d=5,i=2}
		The minimum number of vertices that a non-simplex $3$-simplicial $2$-simple $5$-polytope can have is $12$.
	\end{proposition}
	\begin{proof}
	There exists a $3$-simplicial $2$-simple 5-polytope with $2(5+1)=12$ vertices. Thus, we only need to show that there is no non-simplex $3$-simplicial $2$-simple $5$-polytope with fewer than $12$ vertices.
	
		It is known (see \cite{PafWer}) that every non-simplex $2$-simplicial $2$-simple $4$-polytope has at least $9$ vertices, and the only such polytope with $9$ vertices is the polytope denoted by $P_9$ in \cite{PafWer}. Since vertex figures of $3$-simplicial $2$-simple $5$-polytopes are $2$-simplicial $2$-simple, it follows that a non-simplex $3$-simplicial $2$-simple polytope $Q$ must have at least $10$ vertices. 
		
		Assume that $f_0(Q)=10$. Then each vertex figure is either the $4$-simplex $\sigma_4$ or $P_9$, and so each vertex of $Q$ has degree $5$ or $9$. Since $Q$ is not simple, at least one of the vertex figures of $Q$ is $P_9$. Consider $Q^*$; it has $10$ facets each of which is either $\sigma_4$ or $P_9$. (This is because both $\sigma_4$ and $P_9$ are self-dual.) Now consider a facet $F$ of $Q^*$ that is isomorphic to $P_9$. It has $7$ non-simplex facets (one cross-polytope, also known as an octahedron, and six bipyramids); Construction \ref{constr:P_9}. Each of these seven $3$-faces must lie in $F$ and one additional facet of $Q^*$, which cannot be a simplex. This shows that $Q^*$ has at least eight facets isomorphic to $P_9$. Then in $Q$, at least $8$ out of $10$ vertices are of degree $9$. This implies that all vertices of $Q$ have degree $\geq 8$.	Consequently, all vertices of $Q$ have degree $9$, and so $f_1(Q)=\binom{10}{2}=45$. 
		
		Since $Q$ is $3$-simplicial $2$-simple, $4f_3(Q)=3f_2(Q)$ by Lemma \ref{lm: f-vector criterion}. Furthermore, since $Q$ is $3$-simplicial
		and since the toric $h$-vector of a $5$-polytope is symmetric \cite{Stanley-intersection}, 
	$$	0=g^{\mathrm{toric}}_3(Q)=f_2(Q)-4f_1(Q)+10f_0(Q)-20.$$
		$$\text{Finally, by the Euler relation,}\quad f_0(Q)-f_1(Q)+f_2(Q)-f_3(Q)+f_4(Q)=2.$$
		This uniquely determines the $f$-vector of $Q$: $f(Q)=(10, 45, 100, 75, 12)$. But then we must have $75=f_3(Q)\leq \binom{f_4(Q)}{2}=66$, which is a contradiction.
		
		Similarly, if $f_0(Q)=11$, then $f_2(Q)=4f_1(Q)-10f_0(Q)+20=4f_1(Q)-90$, which is not a multiple of $4$. On the other hand, $4f_3(Q)=3f_2(Q)$ still holds, so $f_3(Q)$  is not an integer, which is again a contradiction.
	\end{proof}
	
	While a $2$-simplicial $2$-simple $4$-polytope with $9$ vertices is unique, this is not the case with $3$-simplicial $2$-simple $5$-polytopes with $12$ vertices. (For instance, in Section 6 we will see that there is such a polytope with a simplex facet.) For $d\geq 6$, Question \ref{Question 1}(1) remains unsolved. It would be very interesting to shed any light on whether the answer is $2d+2$ or smaller than $2d+2$.

	\section{The merging operation}
Throughout, let $d\geq 2$. Recall that a connected sum of two simplicial $d$-polytopes is a simplicial $d$-polytope. In other words, taking connected sums preserves the property of being $(d-1)$-simplicial $1$-simple. Is there an analogous operation that preserves the property of being $(d-i)$-simplicial $i$-simple for an arbitrary $2\leq i\leq d-1$? The goal of this section is to discuss one such operation that can be applied to two $d$-polytopes as long as one of them has a simplex facet and another one has a simple vertex. The order in which we list the vertices will be important for our construction. Specifically, we write $[a_1,\ldots,a_m]$ to denote the polytope $\conv(a_1,\ldots,a_m)$ whose vertices are ordered as $a_1,\ldots,a_m$. We will mainly use this notation to describe faces of a given polytope. For brevity, we also write the edge $[u, v]$ as $uv$. 
	
\subsection{The definition and basic properties}
Let $P_1$ and $P_2$ be two $d$-polytopes such that $P_1$ has a {\em simplex facet} $F:=[u_1, \dots, u_d]$ and $P_2$ has a {\em simple vertex} $v$ whose neighbors are ordered as $u'_1, \dots, u'_d$. We adopt the following notation: for $1\leq j\leq d$, let $H_j$ be the facet of $P_1$ that is adjacent to $F$ along the ridge $G_j:=[u_1, \dots, \widehat{u_j}, \dots, u_ d]$.  Similarly, for $1\leq j\leq d$, let $H'_j$ be the facet of $P_2$ that contains all the edges of $P_2$ incident with $v$ but $vu'_j$.
	
    By applying a projective transformation to $P_1$, we may assume that the hyperplanes $\aff(F), \aff(H_1), \dots, \aff(H_d)$ define a $d$-simplex $\Sigma$ that {\em contains} $P_1$. Denote the vertex of $\Sigma$ that does not lie in $F$ by $u$. By applying the unique affine transformation that maps $v$ to $u$, and $u_k'$ to $u_k$ for $1\leq k\leq d$, we may further assume that the $d$-simplices $\Sigma'=[v, u'_1, \dots, u'_d]$ and $\Sigma$ coincide, and in particular that $P_1\subseteq \Sigma=\Sigma'$ is a convex subset of~$P_2$.
    
	Finally, let $P_2':=\conv(V(P_2)\backslash v)$ and $F':=[u_1', \dots, u_d']$ be two subpolytopes of $P_2$. Note that if $P_2$ is a $d$-simplex, then $P'_2$ is $F'$, and otherwise, $F'$ is a facet of $P'_2$.

	\begin{definition}\label{def: merging}
	Under the above assumptions on $P_1$ and $P_2$, define a new $d$-polytope $P_1\merge P_2$ obtained from $P_2$ by replacing $\Sigma'= \Sigma$ with  $P_1$. Alternatively, $P_1\merge P_2$ is the union of $P_1$ and $P'_2$ where we identify $u_k$ with $u'_k$ for $1\leq k\leq d$. (Observe that $P_1$ and $P_2'$ share the facet $F=F'$, lie on the opposite sides of $F$ and that their union is a polytope.) The new polytope is called the \emph{merge} of $P_1$ and $P_2$ along $F$ and $v$.
	\end{definition}
    \begin{example}
    	Consider two polygons $P_1$ and $P_2$ whose boundary complexes are cycles $(u_1,\ldots,u_n, u_1)$ and $(v_0, v_1,\ldots, v_k, v_0)$. Then the merge of $P_1$ and $P_2$ along the edge $F=u_1 u_n$ and the vertex $v_0$ is the polygon whose boundary complex is the cycle $(v_1=u_1, u_2, \dots, u_{n-1}, u_n=v_k, v_{k-1}, \dots, v_2, v_1=u_1)$. In other words, in dimension $2$, $P_1\merge P_2$ is exactly the connected sum of $P_1$ and $P_2'=\conv(V(P_2)\backslash v_0)$. 
    \end{example}
		
		\noindent Figure 1 illustrates how to merge two $3$-polytopes.
    \begin{figure}[ht]   
    	\subfloat{
    		\begin{tikzpicture}
    			\draw (0,0)node[right]{$u_1$}--(2,2)node[right]{$u_2$}--(1,3)node[right]{$u_3$};
    			\draw (0,0)--(-0.5, 1)--(2, 2)--(0,2.5)--(-0.5,1);
    			\draw (0,2.5)--(1,3);
    			\draw [fill=blue!40, opacity=0.3] (0,0)--(1,3)--(2,2)--(0,0);
    			\draw [blue, dashed] (-0.5, 1)--(-1,2)node[left]{$u$}--(2,2);
    			\draw [blue, dashed](-1, 2)--(0,2.5);
    		\end{tikzpicture}
    	}
    	\subfloat{
    		\begin{tikzpicture}
    			\draw[black] (3,0)--(2,2)node[right]{$u'_2$}--(-1,2)--(0,0)node[left]{$u'_1$}--(3,0)--(5,1);
    			\draw[dashed] (0,0)--(2,1)--(5,1);
    			\draw[dashed] (2,1)--(1,3);
    			\draw (2,2)--(5,1)--(1,3)--(2,2);
    			\draw(-1,2)node[left]{$v$}--(1,3)node[left]{$u'_3$};
    			\draw[fill=blue!40, opacity=0.3] (1,3)--(0,0)--(2,2)--(1,3);
    		\end{tikzpicture}
    	}
    	\subfloat{
    		\begin{tikzpicture}
    			\draw (3,0)--(0,0)node[left]{$u_1=u'_1$}--(-0.5, 1)--(2,2)node[right]{$u_2=u'_2$}--(3,0)--(5,1);
    			\draw[dashed] (0,0)--(2,1)--(5,1);
    			\draw[dashed] (2,1)--(1,3)node[left]{$u_3=u_3'$};
    			\draw (-0.5,1)--(0,2.5)--(2,2)--(1,3)--(0, 2.5);
    			\draw (1,3)--(5,1)--(2,2);
    		\end{tikzpicture}
    	}
    	\caption{$P_1\subseteq \Sigma$, $P_2\supseteq \Sigma'$, and $P_1\merge P_2$, where the merge is along $[u_1, u_2, u_3]\cong [u_1', u_2', u_3']$ and $v$.}
    \end{figure}
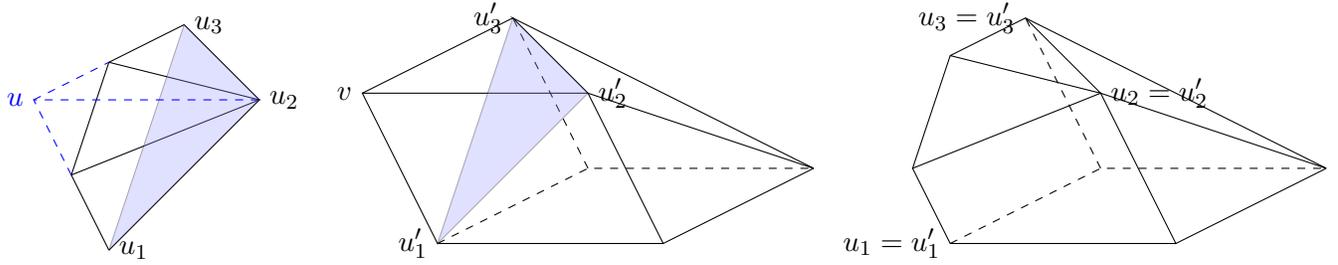\label{Fig: merging}
    
    \begin{remark} \label{rem:facets-of-merge}
   	For $d\geq 3$, the set of facets of $P_1\merge P_2$ consists of 
   		\begin{itemize}
   		\item old facets: all facets of $P_1$ with the exception of  $F, H_1, \dots, H_d$, and all facets of $P_2$ with the exception of $H'_1, \dots, H'_d$;
		\item new facets: for each $1\leq j\leq d$, $H_j$ and $H'_j$ merge into a single facet $H_j\merge H'_j$
			where the merge is along $G_j=[u_1, \dots, \widehat{u_j}, \dots, u_d]$ and $v$ (with the neighbors of $v$ in $H_j'$ ordered as $u_1', \dots, \widehat{u'_j}, \dots u'_d$).
   	\end{itemize}
\end{remark}

\begin{remark} The description of facets of $P_1\merge P_2$ leads to the following observation:
	the combinatorial type of $P_1\merge P_2$ may depend on the ordering of vertices of $F$ and neighbors of $v$. That is, letting $F=[u_{\sigma(1)}, \dots, u_{\sigma(d)}]$ and relabeling the neighbors of $v$ as $v_{\sigma'(1)}, \dots, v_{\sigma'(d)}$, for some permutations $\sigma, \sigma'$ of $[d]:=\{1,2,\ldots, d\}$, may result in a polytope with a different combinatorial type; see Section 6 for examples. This is analogous to the situation with the connected sum of two simplicial polytopes. 
\end{remark}
   It follows from Definition \ref{def: merging} that if $P_1$ is a simplex, then $P_1\merge P_2=P_2$, and similarly if $P_2$ is a simplex, then $P_1\merge P_2=P_1$. In all other cases, $F$ is not a facet of $P_1\merge P_2$ and $v$ is not a vertex of $P_1\merge P_2$. Furthermore, if both $P_1$ and $P_2$ are simplicial and $P_2$ has a simple vertex $v$, then the merge of $P_1$ and $P_2$ along any facet $F$ of $P_1$ and $v$ is the connected sum of $P_1$ and $P_2'=\conv(V(P_2)\backslash v)$.

   We summarize this discussion in the following lemma.
	\begin{lemma}\label{lm: vertices after merging}
	Let $d\geq 2$.	Let $P_1$ be a $d$-polytope with a simplex facet and let $P_2$ be a $d$-polytope with a simple vertex. Then $f_0(P_1\merge P_2)=f_0(P_1)+f_0(P_2)-(d+1)$. In particular, $f_0(P_1\merge P_2)\geq \max \{f_0(P_1), f_0(P_2)\}$ and equality holds if and only if at least one of $P_1$ and  $P_2$ is a simplex. In the case that one of $P_1$ and $P_2$ is a simplex, $P_1\merge P_2$ is equal to the other polytope.
	\end{lemma}
    
    The following theorem and corollary explain the significance of the merging operation.
	\begin{theorem}\label{main theorem}
	Let $d\geq 2$ and $1\leq i,j\leq d-1$, and let $P_1$ and $P_2$ be $d$-polytopes with a simplex facet and a simple vertex, respectively. If $P_1$ and $P_2$ are $j$-simplicial, then so is $P_1\merge P_2$. If $P_1$ and $P_2$ are $i$-simple, then so is  $P_1\merge P_2$.
	\end{theorem}
	\begin{proof}
		We first discuss $j$-simplicial polytopes. The proof is by induction on $d$. The statement holds for $j=1$ for any $d$ (since all polytopes are $1$-simplicial). Hence the statement holds for $d=2$.
		
		Now, assume the statement holds for $d-1$ and any $1\leq j\leq d-2$. We prove that the statement holds for $d$ and any $1\leq j\leq d-1$. Let $P_1$ and $P_2$ be two $j$-simplicial d-polytopes.  If one of them is a simplex, there is nothing to prove. Also, if $j=d-1$, then $P_1\merge P_2$ is the connected sum of two simplicial polytopes $P_1$ and $P'_2$, which is $(d-1)$-simplicial.
		
		Thus assume that $2\leq j\leq d-2$ and that neither $P_1$ nor $P_2$ is a simplex. Let $\tau$ be a $j$-face of $P_1\merge P_2$. Then either $\tau$ is a $j$-face of $P_1$ or it is a $j$-face of $P_2$ or it is  a $j$-face of $H_k\merge H'_k$ for some $k$. In the first two cases, $\tau$ is a simplex because $P_1$ and $P_2$ are $j$-simplicial. In the last case, it is a simplex because both $H_k$ and $H'_k$ are $j$-simplicial, and so $\tau$ is a simplex by the induction hypothesis.
		
		We now discuss $i$-simple polytopes. The proof is again by induction on $d$. The statement holds for $i=1$ and any $d$ (since all polytopes are $1$-simple). Hence the statement holds for $d=2$. Now assume the statement holds for $d-1$ and any $2\leq i\leq d-2$. Let $2\leq i\leq d-1$ and let $P_1$ and $P_2$ be two $i$-simple $d$-polytopes. To see that $P_1\merge P_2$ is $i$-simple, let $\tau$ be a $(d-i-1)$-face of $P_1\merge P_2$. There are two possible cases. 
		
		Case 1: $\tau$ is a face of one of $H_k\merge H'_k$. Since $P_1$ and $P_2$ are $i$-simple, $H_k$ and $H'_k$ are $(i-1)$-simple $(d-1)$-polytopes. Thus, by the induction hypothesis, $H_k\merge H'_k$ is an $(i-1)$-simple $(d-1)$-polytope. Since $\tau$ is a face of $H_k\merge H'_k$ of dimension $d-i-1=(d-1)-(i-1)-1$, it follows that there are exactly $i$ facets of $H_k\merge H'_k$ (and hence ridges of $P_1\merge P_2$) that contain $\tau$. Each of these $i$ ridges is contained in two facets of $P_1\merge P_2$: $H_k\merge H'_k$ and one additional facet. Thus, $\tau$ is contained in exactly $i+1$ facets of $P_1\merge P_2$, namely, $H_k\merge H'_k$ and the $i$ additional facets just described. 
		
		Case 2: $\tau$ is not contained in any $H_k\merge H'_k$ (for $k=1,\ldots,d$). Then either $\tau$ is a face of $P_1$ not contained in any of $F,H_1,\ldots,H_d$, or $\tau$ is a face of $P_2$ that does not contain $v$ and is not contained in any of $H'_1,\ldots,H'_d$. In the former case, the facets of $P_1\merge P_2$ that contain $\tau$ are the facets of $P_1$ that contain $\tau$ and there are $i+1$ of them since $P_1$ is $i$-simple. Similarly, in the latter case, the facets of $P_1\merge P_2$ that contain $\tau$ are the facets of $P_2$ that contain $\tau$ and there are $i+1$ of them. 
	\end{proof}
	\begin{corollary}\label{cor: generate infinite family}
		Let $d\geq 2$ and $1\leq i\leq d-1$. Let $P$ be a $(d-i)$-simplicial $i$-simple $d$-polytope such that (1) $P$ is not a simplex, (2) $P$ has a simplex facet $F$, and (3) $P$ has a simple vertex $v$ not contained in F. Finally, let $P\merge P$ be the merge of $P$ with itself along $F$ and $v$. Then $P\merge P$ is a $(d-i)$-simplicial $i$-simple $d$-polytope that has a simplex facet and a simple vertex  not contained in that facet; furthermore, $f_0(P\merge P)>f_0(P)$.
		Consequently, there exists an infinite family of $(d-i)$-simplicial $i$-simple $d$-polytopes obtained by iterative merging with $P$.
	\end{corollary}
	\begin{proof}
		Consider two copies of $P$: $P_1$ and $P_2$. Denote the copy of $F$ in $P_j$ by $F_j$, and the copy of $v$ in $P_j$ by $v_j$. Merge $P_1$ and $P_2$ along $F_1$ and $v_2$. By Theorem \ref{main theorem}, $P_1\merge P_2$ is $(d-i)$-simplicial and $i$-simple; it has a simplex facet $F_2$ and a simple vertex $v_1\notin F_2$. 
	\end{proof}

This corollary implies that to find infinitely many $(d-i)$-simplicial $i$-simple $d$-polytopes, it suffices to find the ``building blocks" --- those with simplex facets and simple vertices. Hence we propose the following question that strengthens Question \ref{Question 1}(2).
\begin{question} \label{quest:simplex-facet}
Let $d\geq 4$ and $2\leq i\leq d-2$. Are there infinite families of $(d-i)$-simplicial $i$-simple $d$-polytopes, each of which has a {\bf simplex} facet and a {\bf simple} vertex? 
\end{question}

\subsection{The face lattice}
In this subsection, we assume that $P_1$ and $P_2$ are two $(d-i)$-simplicial $i$-simple $d$-polytopes that will be merged along a simplex facet $F=[u_1, \dots, u_d]$ of $P_1$ and a simple vertex $v$ of $P_2$. Our goal is to describe the face lattice of $P_1\merge P_2$, $\lattice(P_1\merge P_2)$. We continue using notation introduced in Section 4.1.

\begin{definition} Consider the following two subposets of $\lattice(P_1)$ and $\lattice(P_2)$:
	$$\lattice(P_1)^-:=\lattice(P_1)\backslash \{\sigma: \sigma \subseteq F, \dim \sigma\geq d-i\},$$
	$$\lattice(P_2)^-:=\lattice(P_2)\backslash \{\sigma: v\in\sigma, \dim \sigma<d-i\},$$
	and let $\lattice(P_1)^-\sqcup \lattice(P_2)^-$ be their {\em disjoint sum}, i.e., the disjoint union of $\lattice(P_1)^-$ and $\lattice(P_2)^-$ with the original partial orders on $\lattice(P_1)^-$ and $\lattice(P_2)^-$, and no other comparable pairs.
\end{definition}
\begin{definition} \label{def:lattice}
	Let $\lattice$ be the following quotient poset of $\lattice(P_1)^-\sqcup \lattice(P_2)^-$ (whose definition depends on $d$ and $i$). As a set, it is $\left(\lattice(P_1)^-\sqcup \lattice(P_2)^-\right)/\sim$, where $$[u_k: k\in S]\sim [u'_k: k\in S] \;\mbox{ for all } S\subseteq [d],\; |S|\leq d-i,$$
	$$\text{and}\; \cap_{k\in S} H_k\sim \cap_{k\in S} H'_k \; \mbox{ for all } S\subseteq [d],\; |S|\leq i.$$
	The partial order on $\lattice$ is inherited from $\lattice(P_1)^-\sqcup \lattice(P_2)^-$: $[\tau]<[\sigma]$ if there are representatives $\tau'$ and $\sigma'$ of the equivalence classes $[\tau]$ and $[\sigma]$  such that $\tau'<\sigma'$ in $\lattice(P_1)^-\sqcup \lattice(P_2)^-$.
\end{definition}
The main result of this subsection ---Theorem \ref{prop: face lattice of the merge}--- asserts that $\lattice$ is the face lattice of $P_1\merge P_2$. The proof relies on the following lemma. 
\begin{lemma}\label{lm: faces of polytopes}
	Let $S\subseteq [d]$. 
	\begin{enumerate}
		\item If $|S|\leq i$, then $\cap_{k\in S}H_k$ is a $(d-|S|)$-face of $P_1$ not contained in $F$, while $\cap_{k\in S} H'_k$ is a $(d-|S|)$-face of $P_2$ containing $v$.
		\item If $|S|\leq d-i$, then $[u_k: k\in S]$ is an $(|S|-1)$-face of $P_1$ and $[u'_k: k\in S]$ is an $(|S|-1)$-face of $P_2$.
		\item If $H$ is a facet of $P_1$ that is not one of $F, H_1,\dots,H_d$, then $H$ shares with $F$ at most $d-i-1$ vertices, and $H$ does not contain any intersection of the form $\cap_{k\in S} H_k$, for $S\subseteq [d]$, $|S|\leq i$.	Hence, $\lattice(H)$ is equal to $[\hat{0}, H]$ computed in both $\lattice(P_1)^-$ and  $\lattice$.
		\item If $H$ is a facet of $P_2$ that does not contain $v$, then $H$ does not contain any intersection of the form $\cap_{k\in S} H'_k$. Thus $\lattice(H)$ is equal to $[\hat{0},H]$ computed in both $\lattice(P_2)^-$ and $\lattice$.
	\end{enumerate}
\end{lemma}
\begin{proof}
	For part (1), we only need to show that $\cap_{k\in S} H_k$ is $(d-|S|)$-dimensional and that it is not contained in $F$. Consider $\tau:=(\cap_{k\in S} H_k)\cap F= \cap_{k\in S}(H_k\cap F)$. Since $F$ is a $(d-1)$-simplex, $\tau$ is a face of $P_1$ of dimension $d-|S|-1$. Now, since $|S|\leq i$, and so $d-|S|-1\geq d-i-1$, the assumption that $P_1$ is $i$-simple implies that the interval $[\tau, \hat{1}]$ is a Boolean lattice whose coatoms are $H_k$, for $k\in S$, and $F$. This, in turn, implies the desired properties of $\cap_{k\in S} H_k$.
	
	For part (2), since $F$ is a simplex facet of $P_1$, $[u_k: k\in S]$ must be a simplex $(|S|-1)$-face of $P_1$. Also, since $v$ is simple, the edges $vu'_k$ for $k\in S$ determine an $|S|$-face of $P_2$,  and this face must be a simplex since $P_2$ is $(d-i)$-simplicial. Thus $[u'_k: k\in S]$ is an $(|S|-1)$-face of $P_2$.
	
	For part (3), note that if $H$ contained $d-i$ vertices of $F$, say, $u_1,\dots,u_{d-i}$, then $[u_1,...,u_{d-i}]$ would be a $(d-i-1)$-face of $P_1$ contained in at least $i+2$ facets, namely, $F$, $H_{d-i+1},\dots,H_d$, and $H$; this is impossible since $P$ is $i$-simple. Similarly, if $H$ contained, say, the face $H_1\cap\dots\cap H_{i}$, then this $(d-i)$-face would be in at least $i+1$ facets, namely, $H_{1},\dots,H_i$, and $H$, which is again a contradiction.
	
	Part (4) follows from the fact that $v\in \cap_{k\in S} H'_k$ but $v\notin H$, and from the definition of $\lattice(P_2)^-$ and $\lattice$.
\end{proof}

 Let $S$ be a subset of $[d]$. Note that $\hat{0}_{P_1}=\vee_{k\in\emptyset}u_k\sim \vee_{k\in\emptyset} u'_k=\hat{0}_{P_2}$ is the minimum element of $\lattice$, while $\hat{1}_{P_1}=\wedge_{k\in\emptyset} H_k\sim \wedge_{k\in\emptyset}H'_k=\hat{1}_{P_2}$ is the maximum element. Furthermore, Lemma \ref{lm: faces of polytopes} implies that if $|S|\leq d-i$, then $\vee_{k\in S} u_k \in\lattice(P_1)$ and $\vee_{k\in S} u'_k\in \lattice(P_2)$ are both elements of $\lattice(P_1)^-\sqcup \lattice(P_2)^-$, and that they have the same rank. Similarly, if $|S|\leq i$, then $\wedge_{k\in S} H_k$ and $\wedge_{k\in S} H'_k$ both belong to $\lattice(P_1)^-\sqcup \lattice(P_2)^-$ and have the same rank there. We are now ready to prove that $\lattice$ is the face lattice of $P_1\merge P_2$. Specifically, for $S\subseteq [d]$, $|S|\leq i$, the class $\wedge_{k\in S}H_k\sim \wedge_{k\in S} H'_k$ in $\lattice$ represents the face $$\cap_{k\in S} (H_k\merge H'_k)=(\cap_{k\in S}H_k)\merge (\cap_{k\in S} H'_k) \mbox{ of $P_1\merge P_2$}.$$ 

\begin{theorem}\label{prop: face lattice of the merge}
	Let $d\geq 2$ and $1\leq i\leq d-1$. Let $P_1$ and $P_2$ be $(d-i)$-simplicial $i$-simple polytopes such that $P_1$ has a simplex facet $F=[u_1,\ldots,u_d]$ and $P_2$ has a simple vertex $v$ whose neighbors are $u'_1,\ldots,u'_d$. Then $\lattice=\lattice(P_1\merge P_2)$.
\end{theorem}
\begin{proof}
The proof is by induction on $d$ and $i$. First we consider the case where $P_1$ and $P_2$ are both $(d-1)$-simplicial $1$-simple $d$-polytopes. This case splits into two subcases:
\begin{enumerate}
	\item If $P_2$ is not a simplex, then $P_1\merge P_2=P_1\# P'_2$. The lattice $\lattice(P_1\merge P_2)$ is obtained from $\lattice(P_1)$ and $\lattice(P'_2)$ by removing facets $[u_1,\ldots,u_d]$ and $[u'_1,\ldots,u'_d]$ and identifying their boundary complexes; this agrees with our definition of $\lattice(P_1)^-\sqcup \lattice(P_2)^-/\sim \ = \lattice$.
	\item If $P_2$ is a simplex, then $P_1\merge P_2$ is $P_1$. That $\lattice$ is equal to $\lattice(P_1)$ in this case, again follows easily from the definition of $\lattice$.
\end{enumerate} 

\noindent This discussion completes the proof of the base case $i=1$ and arbitrary $d\geq 2$. 

Now assume that the statement holds in dimension $\leq d-1$ and consider two $(d-i)$-simplicial $i$-simple $d$-polytopes $P_1$ and $P_2$, where $i\geq 2$. By definition, $\lattice$ and $\lattice(P_1\merge P_2)$ have the same coatoms. So it suffices to show that for every facet $H$ of $P_1\merge P_2$,  the interval $[\hat{0}, H]$ in $\lattice$ is equal to $\lattice(H)$.

First, if $H$ is a facet of $P_1$ not equal to $F, H_1,\dots,H_d$, or $H$ is a facet of $P_2$ that does not contain $v$, then by Lemma \ref{lm: faces of polytopes}, the interval $[\hat{0}, H]$ in $\lattice$ is equal to $\lattice(H)$. For $1\leq k\leq d$,  both $H_k$ and $H'_k$ are $(d-i)$-simplicial $(i-1)$-simple $(d-1)$-polytopes. In particular,
\begin{eqnarray*}
\lattice(H_k)^-&=&\lattice(H_k)\backslash \{\sigma: \sigma \subseteq F\backslash u_k, \;\dim \sigma\geq (d-1)-(i-1)=d-i\},\\
 \lattice(H'_k)^-&=&\lattice(H'_k)\backslash \{\sigma: v\in\sigma,\; u'_k\notin\sigma,\; \dim \sigma<(d-1)-(i-1)=d-i\}.
\end{eqnarray*} 
Hence $[0,H_k]$ computed in $\lattice(P_1)^-$ is $\lattice(H_k)^-$ and  
$[0,H'_k]$ computed in $\lattice(P_2)^-$ is $\lattice(H'_k)^-$. Then the inductive hypothesis implies that $[\hat{0}, H_k\merge H'_k]$ in $\lattice$ is equal to $\lattice(H_k\merge H'_k)$. This proves that $\lattice=\lattice( P_1\merge P_2)$.
\end{proof}

One application of Theorem \ref{prop: face lattice of the merge} is the following result on the $f$-numbers of $P_1\merge P_2$.

 \begin{corollary}\label{prop: f-vector}
	Let $d\geq 2$ and $1\leq i\leq d-1$. Let $P_1$ and $P_2$ be $(d-i)$-simplicial $i$-simple $d$-polytopes that can be merged along a simplex facet $F$ of $P_1$ and a simple vertex $v$ of $P_2$. Then for all $0\leq j\leq d-1$, $f_j(P_1\merge P_2)=f_j(P_1)+f_j(P_2)-\binom{d+1}{j+1}$.
\end{corollary}
\begin{proof}
	First assume that $0\leq j\leq d-i-1$. By definition of $\lattice(P_1\merge P_2)$, each $j$-face of $F$ (i.e., each $(j+1)$-subset of $\{u_1,\ldots, u_d\}$), is identified with the corresponding  $j$-face of $F'$ (i.e., the corresponding $(j+1)$-subset of $\{u'_1,\ldots, u'_d\}$). In addition, all $j$-faces of $P_2$ that contain $v$ (i.e., all $(j+1)$-subsets of $\{v, u'_1,\ldots,u'_d\}$ that contain $v$) are removed from $\lattice(P_1\merge P_2)$. Hence 
	$$f_j(P_1\merge P_2)=f_j(P_1)+f_j(P_2)-\binom{d}{j+1}-\binom{d}{j}=f_j(P_1)+f_j(P_2)-\binom{d+1}{j+1}.$$

	Similarly, for $d-i\leq j\leq d-1$, by definition of $\lattice(P_1\merge P_2)$, all $j$-faces of $P_1$ contained in $F$ (i.e., 
	$(j+1)$-subsets of $\{u_1,\ldots,u_d\}$) are removed from $\lattice(P_1\merge P_2)$, while for each $(d-j)$-subset $S$ of $[d]$, the $j$-face $\cap_{k\in S} H_k$ is identified with the $j$-face $\cap_{k\in S} H'_k$. Hence $f_j(P_1\merge P_2)=f_j(P_1)+f_j(P_2)-\binom{d}{j+1}-\binom{d}{d-j}=f_j(P_1)+f_j(P_2)-\binom{d+1}{j+1}. $ 
\end{proof}

\section{Applications: part I}
\subsection{Infinite families of $(d-i)$-simplicial $i$-simple polytopes for small $d$}
The goal of this section is to  answer Question \ref{quest:simplex-facet} in the affirmative for small values of $d$.
Our starting point is the uniform 8-polytope $2_{41}$ constructed within the symmetry of the $E_8$ group. (It was first discovered by Gosset and Elte; see also \cite[Section 11]{Coexter}). This polytope has $17280$ simplex facets and it is $4$-simplicial and $4$-simple. The polytope $2_{41}$ gives rise to the following $7$-polytopes:
\begin{itemize}
	\item Each nonsimplex facet of $2_{41}$ is the $7$-polytope $2_{31}$. It is $4$-simplicial $3$-simple and it has $576$ simplex facets.
	\item Each vertex figure of $2_{41}$ is the $7$-demicube.
\end{itemize}

Recall that the {\em $d$-demicube} is defined as follows (see \cite[Exercise 4.8.18]{Gru-book}). Consider the $d$-cube $C_d=[0,1]^d$. For each vertex $v$ in $C_d$ whose coordinates have an even number of ones, truncate $C_d$ along the hyperplane that contains all $d$ vertices adjacent to $v$. The resulting polytope is called the $d$-demicube; we denote it by $Q_d$. This polytope has the following  properties:
\begin{itemize}
	\item When $d>4$, $Q_d$ has exactly $2^{d-1}$ simplex facets (these are the facets defined by truncating hyperplanes), 
	and $2d$ non-simplex facets (these are the facets obtained by truncating the facets of $C_d$). Moreover, no two simplex facets are adjacent in $Q_d$. 
	\item When $d\geq 4$, $Q_d$ is $3$-simplicial and $(d-3)$-simple.
\end{itemize}

We are now in a position to prove the main result of this subsection:
\begin{theorem}\label{thm: lower dimensional infinite families}
	For every element of $\{(i,d) : 2\leq i\leq d-2\leq 6\}\backslash\{(3,8), (5,8)\}$, there exists an infinite family of $(d-i)$-simplicial $i$-simple $d$-polytopes, each of which has a simplex facet and a simple vertex not in that facet.
\end{theorem}
\begin{proof}
By considering dual polytopes, it suffices to prove the statement for $i\leq d/2\leq 4$. The case of $i=2$ and an arbitrary $d\geq 4$ will be discussed in Section 6. For now, we mention that for $i=2$ and $d=4$, the result follows by applying Corollary \ref{cor: generate infinite family} to $P_9$. (For the description of facets of $P_9$, see Construction \ref{constr:P_9}.)
	Consider the case of $i=3$ and $d=6$. Since both $Q_6$ and $Q^*_6$ are $3$-simplicial $3$-simple, and since $Q_6$ has a simplex facet (in fact, $32$ of them) and $Q^*_6$ has a simple vertex (in fact, $32$ of them), the merge of $Q_6$ and $Q^*_6$, $P=Q_6\merge Q^*_6$, is well-defined; furthermore, $P$ has a simplex facet $F$ and a simple vertex $v$ not contained in $F$. Hence, Corollary \ref{cor: generate infinite family} applies to $P$ and results in a desired infinite family of $3$-simplicial $3$-simple $6$-polytopes. Similarly, in the case of $i=3$ and $d=7$, apply Corollary \ref{cor: generate infinite family} to $P=2_{31}\merge Q_7^*$. Finally, in the case of $i=4$ and $d=8$, apply Corollary \ref{cor: generate infinite family} to $P=2_{41}\merge 2_{41}^*$.
\end{proof}

The proof of Theorem \ref{thm: lower dimensional infinite families} provides the following partial answer to Question \ref{quest:simplex-facet}.
\begin{corollary} \label{cor:i-simplicial-2i-simplex-facet}
Let $2\leq i\leq 4$. There exists an infinite family of $i$-simplicial $i$-simple $2i$-polytopes, each of which has a simplex facet and a simple vertex not in that facet.
\end{corollary}

\subsection{Self-dual polytopes}
Kalai \cite[Problem 19.5.24]{Kalai-skeletons} asked for which values of $i$ and $d$ there are self-dual $i$-simplicial $d$-polytopes other than the $d$-simplex. For the rest of this section, assume that $d=2i$ and consider an $i$-simplicial $i$-simple $2i$-polytope $P$ with a simplex facet $F=[u_1, \dots, u_{2i}]$. As before, assume that $H_1, \dots, H_d$ are the facets of $P$  adjacent to $F$, where $H_k\cap F=[u_1, \dots, \widehat{u_k}, \dots, u_d]$.  Let $\phi: \lattice(P)\to \lattice(P^*)$, $\phi: \lattice(P^*)\to \lattice(P)$ be the order-reversing bijections on the face lattices. Then $P^*$ is an $i$-simplicial $i$-simple $2i$-polytope with a simple vertex $v:=\phi(F)$. The neighbors of $v$ are $u'_k:=\phi(H_k)$ for $1\leq k\leq d$. Let $H'_k$ be the facet of $P^*$ determined by the edges $vu'_1, \dots, \widehat{vu'_k}, \dots, vu'_d$. In other words, $H'_k=(\vee_{j\in [d]\backslash k} u'_j)\vee v$, and hence $$\phi(H'_k)=\left(\wedge_{j\in [d]\backslash k}\phi(u'_j)\right)\wedge \phi(v)=\left(\wedge_{j\in [d]\backslash k}H_j\right)\wedge F=u_k.$$

The next proposition is our main tool for constructing self-dual $i$-simplicial $i$-simple $2i$-polytopes. We follow assumptions and notation introduced in the previous paragraph.
\begin{proposition}\label{prop: self-dual}
	The merge of $P$ and $P^*$ along $F=[u_1, \dots, u_d]$ and $v$ (whose neighbors are ordered as $u'_1, \dots, u'_d$) is a self-dual polytope.
\end{proposition}
\begin{proof}
	The map $\phi: \lattice(P)\to \lattice(P^*), \lattice(P^*)\to \lattice(P)$ provides us with an order-reversing involution on $\lattice(P)\sqcup \lattice(P^*)$. Since $\phi(H_k)=u'_k$ and $\phi(H_k')=u_k$, it follows that for $S\subseteq [d]$, 
	\begin{equation}\label{eq1}
		\phi(\vee_{k\in S} u_k)=\wedge_{k\in S} H'_k,\quad  \phi(\vee_{k\in S}u'_k)=\wedge_{k\in S}H_k.
	\end{equation}
	In particular, $\phi$ maps $\ell$-faces of $F$ to $(d-\ell-1)$-faces containing $v$. Since $d=2i$, it follows that $\phi$ induces an order-reversing involution on $\lattice(P)^-\sqcup \lattice(P^*)^-$. Furthermore, by (\ref{eq1}), this involution descends to an order-reversing involution on the quotient $\lattice$ described in Definition \ref{def:lattice}. Thus $\lattice$ is a self-dual lattice. The result follows since by Theorem \ref{prop: face lattice of the merge}, $\lattice=\lattice(P\merge P^*)$.
\end{proof}
\begin{theorem} \label{thm:inf-many-self-dual}
For all $2\leq i\leq 4$, there exists an infinite family of self-dual $i$-simplicial $2i$-polytopes.
\end{theorem}
\begin{proof}
Let $2\leq i\leq 4$. By Corollary \ref{cor:i-simplicial-2i-simplex-facet}, there exists an infinite family of  $i$-simplicial $i$-simple $2i$-polytopes each of which has a simplex facet. The result follows by applying Proposition \ref{prop: self-dual} to this family.
\end{proof}

\section{Applications: part II}
This section is devoted to $(d-2)$-simplicial $2$-simple $d$-polytopes for all $d\geq 4$. We show that for such values of parameters, the answer to Question \ref{quest:simplex-facet} is yes, and, in fact, that  for every $d\geq 4$, there are $2^{\Omega(N)}$ combinatorial types of $(d-2)$-simplicial $2$-simple $d$-polytopes with at most $N$ vertices, each of which has a simplex facet and a simple vertex. Section 6.1 concentrates on a few constructions for $d=4$; Section 6.2 treats the general case.

\subsection{Revisitng $2$-simplicial $2$-simple $4$-polytopes}
By a result of Paffenholz and Werner \cite{PafWer}, there exist infinite families of $2$-simplicial $2$-simple $4$-polytopes each of which has a simplex facet and a simple vertex. This solves Question \ref{quest:simplex-facet} in the affirmative in dimension $d=4$.
	
In this section, we provide alternative (and more symmetric) constructions. We start by revisiting the construction from \cite{PafWer} of $P_9$ --- the unique $2$-simplicial $2$-simple $4$-polytope with nine vertices --- casting it in a way that will help us construct higher-dimensional analogs of $P_9$  in Section 6.2. We then provide another construction of a highly symmetric $2$-simplicial $2$-simple $4$-polytope with $18$ vertices that appears to be new. The promised infinite families are obtained by merging $k$ copies of $P_9$ (respectively, $P_{18}$) for all natural numbers $k\geq 2$. The cross-polytope is featured prominently in our constructions, and we often abbreviate it as $\CP$. (The notion of a {\em point beyond or beneath a facet} is defined in \cite[page 78]{Gru-book}.)

\begin{construction} \label{constr:P_9}
To construct $P_9$, start with a regular $4$-simplex $\Sigma:=[u'_1,u'_2,u'_3,u'_4,u'_5]$. Now add the vertices $u_1,u_2,u_3,v_2$ in the following way. (Why we label the vertices in this fashion will become clear in Section 6.2.) For $i=1,2,3$, place $u_i$ in the affine hull of the facet $\Sigma\backslash u'_i$ of $\Sigma$ so that it is positioned beyond the $2$-face $\Sigma\backslash u'_iu'_5$ and so that $[u_1, u_2, u_3, u'_1, u'_2, u'_3]$ is a 3-cross-polytope; cf.~Definition \ref{def: the first polytope P_1} below. (Hence $u_i$ can be thought of as a perturbation of the barycenter of $[u'_j,u'_k, u'_\ell]$, where $\{i, j, k, \ell\}=[4]$.) Then position $v_2$ on the intersection of the affine hulls of $[u'_1,u'_4,u_2,u_3]$, $[u'_2,u'_4,u_1,u_3]$, and $[u'_3,u'_4,u_1,u_2]$ (this intersection is a line) and beyond the hyperplane $\aff(u'_4,u_1,u_2,u_3)$;  cf.~Definitions \ref{def:a_i} and \ref{def: (d-2)-simplicial 2-simple d-polytope}. (Thus, $v_2$ is a special perturbation of the barycenter of $[u_1,u_2,u_3,u_4']$). 
 
The resulting polytope has nine vertices $\{v_2, u_1,u_2, u_3, u'_1, \ldots, u'_5\}$; it is also convenient to let $v_1=u'_4$. Figure \ref{Fig: P_9} shows part of the Schlegel diagram of $P'_9=\conv(V(P_9)\backslash u'_5)$. The complete list of facets of $P_9$ is given as follows (cf.~Lemma \ref{lm: list of facets}):

	\begin{enumerate}
		\item a $\CP$ with antipodal facets $[u_1, u_2, u_3]$ and $[u'_1, u'_2, u'_3]$ (colored in blue) and a simplex $[u'_1, u'_2, u'_3, u'_5]$;
		\item three bipyramids $[u_1, u'_5, u'_2, u'_3, u'_4]$, $[u_2, u'_5, u'_1, u'_3, u'_4]$, and $[u_3, u'_5, u'_1, u'_2, u'_4]$, where the pairs of suspension vertices are  $(u_1, u'_5)$, $(u_2, u'_5)$, and $(u_3, u'_5)$, respectively;		
		\item three more bipyramids $[v_2, u'_1, u_2, u_3, v_1]$ (colored in purple), $[v_2, u'_2, u_1,u_3, v_1]$, and $[v_2, u'_3, u_1, u_2, v_1]$, where the pairs of suspension vertices are  $(v_2, u'_1)$, $(v_2, u'_2)$, and $(v_2, u'_3)$, respectively;	
		\item another simplex $[v_2, u_1, u_2, u_3]$ (colored in orange).
	\end{enumerate}
\end{construction} 
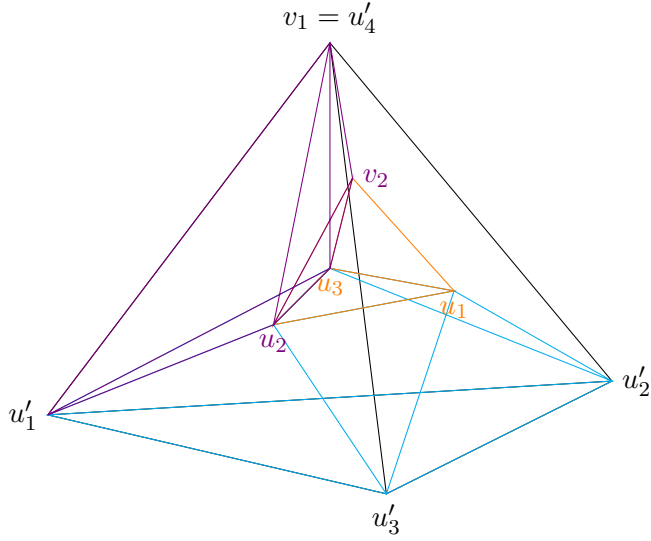
\begin{figure}[ht]
	\centering
\begin{tikzpicture}[scale=1.5]
			\draw (0,0)node[below]{$u'_3$}--(2,1)node[right]{$u'_2$}--(-0.5, 4)node[above]{$v_1=u'_4$}--(-3, 0.7)node[left]{$u'_1$}--(0,0)--(-0.5, 4);
			\draw (-3, 0.7)--(2,1);
			\draw[cyan] (-1, 1.5)--(0.6, 1.8)--(-0.5, 2)--(-1,1.5)--(0,0)--(0.6, 1.8)--(2,1)--(0,0)--(-3, 0.7)--(-1, 1.5)--(-0.5, 2)--(2,1)--(-3, 0.7)--(-0.5, 2);
			\draw [orange] (-1, 1.5)--(0.6, 1.8)node[below]{$u_1$}--(-0.5, 2)node[below]{$u_3$}--(-0.3, 2.8)--(-1, 1.5)--(-0.5, 2);
			\draw[orange] (-0.3, 2.8)--(0.6, 1.8);
			\draw[violet] (-3, 0.7)--(-0.5, 4)--(-0.3, 2.8)node[right]{$v_2$}--(-0.5, 2)--(-1, 1.5)--(-3, 0.7)--(-0.5, 2)--(-0.5, 4)--(-1, 1.5)node[below]{$u_2$}--(-0.3, 2.8);
		\end{tikzpicture}
    \caption{Parts of the Schlegel diagrams of $P'_9$.}
    \label{Fig: P_9}
\end{figure}

The list of facets shows that $P_9$ is $2$-simplicial. The $f$-vector of $P_9$ is symmetric, namely, $f(P_9)=(9,26,26,9)$. Thus, by Corollary \ref{cor:equality}, $P_9$ is also $2$-simple. Furthermore, $P_9$ has two pairs of a simplex facet and a simple vertex not in that facet: $([v_2, u_1, u_2, u_3], u'_5)$ and $([u_1', u_2', u_3', u_5'],v_2)$. Take two copies of $P_9$,  $P^l_9$ and $P^r_9$, and consider the merge $P^l_9\merge P^r_9$ along $[v_2, u_1, u_2, u_3]$ from $P^l_9$ and $u'_5$ from $P^r_9$. Since the facets of $P_9$ containing $u'_5$ consist of a simplex and three bipyramids, depending on the order in which we list the neighbors of $u'_5$, the cross-polytopal facet of $P_9^l$ will either be merged with a $3$-simplex or with a bipyramid of $P_9^r$, resulting in two distinct combinatorial types of $2$-simplicial $2$-simple $4$-polytopes, each of which has a simplex facet and a simple vertex not in that facet. This observation will allow us to construct exponentially many (in the number of vertices) $2$-simplicial $2$-simple $4$-polytopes. We will return to this discussion (and provide many more details) in Section 6.2 after we construct a $d$-dimensional analog of $P_9$ for all $d\geq 4$; see Theorem \ref{prop: counting combinatorial types} and Remark \ref{rem:exp-many-d=4}.

How does merging with $P_9$ affect the $f$-numbers? Let $Q$ be a $2$-simplicial $2$-simple $4$-polytope that has a simplex facet and a simple vertex not in this facet (for instance, $Q=P_9$). Then $P_9\merge Q$ and $Q\merge P_9$ are both defined and by Corollary \ref{prop: f-vector}, 
\begin{eqnarray*}f(P_9\merge Q)- f(Q) =f(Q\merge P_9) - f(Q)&=&f(P_9)-\left(\binom{5}{1}, \binom{5}{2},\binom{5}{3}, \binom{5}{4}\right)\\
&=&(9,26,26,9)-(5,10,10,5)=(4, 16, 16, 4).\end{eqnarray*}
Recall that the toric $g_2$-number of a $2$-simplicial $4$-polytope is given by $g_2^{\rm{toric}}=f_1-4f_0+10$ and that any polytope with $g_2^{\rm{toric}}=0$ is called an \emph{elementary} polytope. It then follows that $P_9$ is an elementary polytope and that $g_2^{\rm{toric}}(P_9\merge Q)=g_2^{\rm{toric}}(Q\merge P_9)=g_2^{\rm{toric}}(Q)$. In other words, if $Q$ is also an elementary polytope, then so are  $P_9\merge Q$ and $Q\merge P_9$. (Elementary polytopes play an important role in the Lower Bound Theorem, see \cite{Kalai87}.)

It is worth pointing out that if one applies to $Q$ the second construction from \cite[Section 3.2]{PafWer}, the resulting polytope $\mathcal{I}^2(Q)$ has the same $f$-vector as $f(P_9\merge Q)=f(Q\merge P_9)$; see \cite[Theorem 3.7]{PafWer}. At the same time, 
both polytopes $P_9\merge Q$ and $Q\merge P_9$ are different from $\mathcal{I}^2(Q)$. Indeed, merging with $P_9$, on the left or on the right, always generates a facet (contributed by the cross-polytopal facet of $P_9$) that is isomorphic to either $\CP$ or the connected sum of $\CP$ with another $3$-polytope, while in the second construction of \cite{PafWer}, all new facets are stacked $3$-polytopes with either $4$, $5$, or $6$ vertices. 

Our next task is to describe another highly-neighborly $2$-simplicial $2$-simple $4$-polytope with a simplex facet and a simple vertex. This polytope has $18$ vertices and we denote it by $P_{18}$.
\begin{construction}
    We start with a regular $3$-simplex $F=[v_1,v_2,v_3,v_4]$ in $\R^3\times \{0\} $. Specifically, let \begin{equation}
		v_1=(0,0,0,0), \, v_2=(2,2, 0,0), \, v_3=(2,0,2,0),\, v_4=(0,2,2,0).
	\end{equation}
    Define $u=(1,1,1,h)$ for some $h>0$. Let $0<\epsilon\ll 1$. For all distinct $1\leq i,j,k\leq 4$, let $$u_{ji, k}=u_{ij, k}=\frac{1}{2}(v_i+v_j)+\epsilon(u+v_k-v_i-v_j).$$ That is, 
		$$u_{12,3}=(1+\epsilon, 1-\epsilon, 3\epsilon, h\epsilon), \, u_{12,4}=(1-\epsilon, 1+\epsilon, 3\epsilon, h\epsilon),\,  u_{13,2}=(1+\epsilon,3\epsilon, 1-\epsilon, h\epsilon),$$
 $$u_{13,4}=(1-\epsilon, 3\epsilon, 1+\epsilon, h\epsilon),\, u_{14,2}=(3\epsilon, 1+\epsilon, 1-\epsilon, h\epsilon), \, u_{14,3}=(3\epsilon, 1-\epsilon, 1+\epsilon, h\epsilon), $$
$$u_{23,1}=(2-3\epsilon, 1-\epsilon, 1-\epsilon, h\epsilon), \,  u_{23,4}=(2-3\epsilon, 1+\epsilon, 1+\epsilon, h\epsilon),\, u_{24,1}=(1-\epsilon,2-3\epsilon,1-\epsilon, h\epsilon), $$
  $$ u_{24,3}=(1+\epsilon, 2-3\epsilon, 1+\epsilon, h\epsilon),\, u_{34,1}=(1-\epsilon, 1-\epsilon, 2-3\epsilon, h\epsilon),\, u_{34,2}=(1+\epsilon, 1+\epsilon, 2-3\epsilon, h\epsilon).$$

    Note that each $u_{ij, k}$ can be viewed as a certain perturbation of the barycenter of $[v_i,v_j]$ that keeps it in the hyperplane defined by $[u, v_i, v_j, v_k]$.  Note also that the set of vertices $\{u_{1i, j}: \{i, j\}\in \{2,3,4\}\}$ forms a hexagon $H_1$ that lies in the affine $2$-space $x_1+x_2+x_3=2+3\epsilon, x_4=h\epsilon.$ Similarly, the sets of vertices $$\{u_{2i, j}: \{i, j\}\in \{1,3,4\}\}, \,\{u_{3i, j}: \{i, j\}\in \{1,2,4\}\},\mbox{ and }\{u_{4i, j}: \{i, j\}\in \{1,2,3\}\}$$  form hexagons $H_2, H_3, H_4$ in the planes defined by 
		\begin{eqnarray*}\{x_1+x_2-x_3=2-3\epsilon,\; x_4=h\epsilon\}, \quad \{x_1-x_2+x_3=2-3\epsilon,\; x_4=h\epsilon\}, \quad \mbox{and} \\
		\{-x_1+x_2+x_3=2-3\epsilon, \; x_4=h\epsilon\},\end{eqnarray*}
    respectively. It follows that 
		\begin{eqnarray*}
    \aff(v_1\cup H_1)&=&\{{\mathbf x}\in\R^4 \ : \ -h\epsilon(x_1+x_2+x_3)+(2+3\epsilon)x_4=0\}, \\
    \aff(v_2\cup H_2)&=&\{{\mathbf x}\in\R^4 \ : \  h\epsilon(x_1+x_2-x_3)+(2+3\epsilon)x_4=4h\epsilon\},\\
    \aff(v_3\cup H_3)&=&\{{\mathbf x}\in\R^4 \ : \ h\epsilon(x_1+x_3-x_2)+(2+3\epsilon)x_4=4h\epsilon\},\\
    \aff(v_4\cup H_4)&=&\{{\mathbf x}\in\R^4 \ : \ h\epsilon(x_2+x_3-x_1)+(2+3\epsilon)x_4=4h\epsilon\}.
		\end{eqnarray*}
    The intersection of these four hyperplanes is the point $(1,1,1,\frac{3h\epsilon}{2+3\epsilon})$; we denote it by $w$.
\end{construction}
Define $P'_{18}$ as the convex hull of all $17$ vertices $\{w, v_1, \dots, v_4, u_{ij, k}: 1\leq i, j, k\leq 4\}$. When $\epsilon$ is very small, the polytope $P'_{18}$ has the following $19$ facets (see Figure \ref{Fig: P_18} for part of the Schlegel diagram). We used $\epsilon=0.05$, $h=2$ and verified this list with software SAGE. 
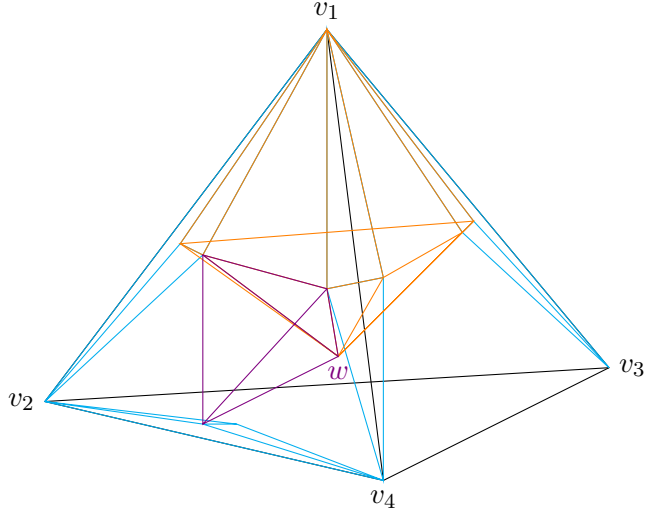
\begin{figure}\centering
	\begin{tikzpicture}[scale=1.5]
		\draw (0,0)node[below]{$v_4$}--(2,1)node[right]{$v_3$}--(-0.5, 4)node[above]{$v_1$}--(-3, 0.7)node[left]{$v_2$}--(0,0)--(-0.5, 4);
		\draw (-3,0.7)--(2,1);
		\draw[cyan] (2,1)--(-0.5, 4)--(-3,0.7)--(0,0);
		\draw[cyan] (-1.6, 2)--(-0.5, 4)--(-1.8, 2.1)--(-3, 0.7)--(-1.6, 2)--(-1.8, 2.1);
		\draw[cyan] (-0.5, 1.7)--(-0.5, 4)--(0, 1.8)--(0,0)--(-0.5, 1.7)--(0, 1.8);
		\draw[cyan] (0.7, 2.2)--(2,1)--(0.8, 2.3)--(-0.5, 4)--(0.7, 2.2)--(0.8, 2.3);
		\draw[cyan] (-1.6, 0.5)--(0,0)--(-1.3, 0.5)--(-3, 0.7)--(-1.6, 0.5)--(-1.3, 0.5);
		\draw[orange] (-1.8, 2.1)--(-1.6, 2)--(-0.5, 1.7)--(0, 1.8)--(0.7, 2.2)--(0.8, 2.3)--cycle;
		\draw[orange](-0.5, 4)--(-1.8, 2.1)--(-0.4, 1.1)--(-1.6, 2)--(-0.5, 4)--(-0.5, 1.7)--(-0.4, 1.1)--(0, 1.8)--(-0.5, 4)--(0.7, 2.2)--(-0.4, 1.1)--(0.8, 2.3)--(-0.5, 4);
		\draw[violet] (-1.6, 2)--(-0.5, 1.7)--(-1.6, 0.5)--(-1.6, 2)--(-0.4, 1.1)node[below]{$w$}--(-1.6, 0.5);
		\draw[violet](-0.4, 1.1)--(-0.5, 1.7);
	\end{tikzpicture}
 \caption{Parts of the Schlegel diagrams of $P'_{18}$.}
\label{Fig: P_18}
\end{figure}
 \begin{enumerate}
	\item Six simplices of the form $[v_i, v_j, u_{ij, k}, u_{ij, m}]$, where $\{i, j, k, m\}=[4]$. Parts of four of them are shown in blue in Figure \ref{Fig: P_18}.
	\item Four simplices of the form $[u_{ij, k}, u_{ik, j}, u_{jk, i}, w]$, where $1\leq i, j, k\leq 4$ are distinct. One such simplex is shown in purple in Figure \ref{Fig: P_18}.
	\item The simplex $[v_1, v_2, v_3, v_4]$.
	\item Four polytopes of the form $[v_i, w, u_{ij, k}, u_{ij, m}, u_{ik, j}, u_{ik, m}, u_{im, j}, u_{im, k}]$. Each is the suspension over $H_i$, with suspension vertices $v_i$ and $w$. (Here $\{i, j, k, m\}=[4]$.) One such polytope is shown in orange in Figure \ref{Fig: P_18}.
	\item Four cross-polytopes of the form $[v_i, v_j, v_k, u_{ij, k}, u_{ik, j}, u_{jk, i}]$, where $1\leq i, j, k\leq 4$ are distinct.
\end{enumerate}

To complete the construction of $P_{18}$,  we apply a projective transformation $\pi$ to $P'_{18}$ to ensure that the adjacent facets of $G=[v_1,v_2, v_3, v_4]$, i.e., the four cross-polytopes from the last item, intersect at a point $w'$ beyond $G$. We let $P_{18}=\conv(\pi(P'_{18})\cup w')$. Then $G$ is not a facet of $P_{18}$ and each facet $[v_i, v_j, v_k, u_{ij, k}, u_{ik, j}, u_{jk, i}]$ is replaced by its connected sum with $[v_i, v_j, v_k, w']$. It can be checked that $f(P_{18})=(18,64,64,18)$. Since $P_{18}$ is a $2$-simplicial $4$-polytope that has $f_1=f_2$, it follows by Corollary \ref{cor:equality} that $P_{18}$ is also $2$-simple. A direct computation shows that $g^{\mathrm{toric}}_2(P_{18})=2$. In other words, $P_{18}$ is not elementary.

Observe that $P_{18}$ has a simple vertex $w'$ and many simplex facets not containing $w'$ (see the first item in the list). Thus we can iteratively merge $P_{18}$ with itself and obtain an infinite sequence of $2$-simplicial $2$-simple $4$-polytopes, each having at least one simplex facet and one simple vertex. By Corollary \ref{prop: f-vector}, any polytope obtained by merging $k\geq 1$ copies of $P_{18}$ will have $5+13k$ vertices and $g^{\mathrm{toric}}_2=2k$. Other families of $2$-simplicial $2$-simple $4$-polytopes where the $k$th polytope has $g_2^{\mathrm{toric}}=2k$ (but $f_0=10+4k$) were constructed in \cite[Corollary 4.2]{PafZieg}.

To close this section, we propose the following problem. 
\begin{question}
	Is there a sequence of $2$-simplicial $2$-simple $4$-polytopes that approximate the unit ball?
\end{question}
\noindent In light of \cite[Theorem 3.2]{ANS}, it is natural to conjecture that if such a sequence of $4$-polytopes 
$\{Q_i\}$ exists, then $\lim_{i\to \infty} g_2^{\mathrm{toric}}(Q_i)=\infty$.

\subsection{Many $(d-2)$-simplicial $2$-simple $d$-polytopes}
In this section we construct a $d$-dimensional analog of $P_9$ for all $d\geq 4$. We then use this polytope along with Corollary \ref{cor: generate infinite family} to show that there are $2^{\Omega(N)}$ combinatorial types of $(d-2)$-simplicial $2$-simple $d$-polytopes with at most $N$ vertices and an additional property that each of these polytopes has a simplex facet and a simple vertex.  

As in Section 6.1, the $d$- and $(d-1)$-dimensional cross-polytopes are used  frequently, and we abbreviate them as CP. To start, we introduce the notion of a pseudo-regular CP and prove some of its properties. Let $\mathbf{0}$ denote the origin of $\R^{d-1}$.

\begin{definition}\label{def: associated points}
Let $G\subset \R^{d-1}$ be a regular $(d-1)$-simplex centered at the origin, let $G^*\subset \R^{d-1}$ be the dual of $G$, and let $\alpha > 0$ be a real number. Assume also that $G$ is contained in the interior of $\alpha G^*$, denoted $\inter(\alpha G^*)$.
	A $d$-cross-polytope is called \emph{pseudo-regular} if it is congruent to $\conv (G\times \{1\}  \cup   \alpha G^*\times \{-1\}).$
\end{definition}

Consider a regular simplex $G=[\mu_1, \dots, \mu_d]\subset \R^{d-1}$ centered at the origin and let $\alpha>0$. Then $\alpha G^*=[\mu'_1,\dots, \mu'_d]\subset \R^{d-1}$ is also a regular simplex centered at the origin. We label the vertices in such a way that $\mu'_i$ is an outer normal vector to the facet $[\mu_1, \dots, \widehat{\mu_i}, \dots, \mu_d]$ of $G$. By our assumptions on $G$, this is equivalent to labeling the vertices so that for all $i\in[d]$, $\mu'_i=a\sum_{j\in[d]\backslash i}\mu_j=-a\mu_i$, where $a$ is a positive scalar independent of $i$.

For a nonempty subset $I$ of $[d]$, let $G_I=[\mu_i : i\in I]$ be a face of $G$ and $G'_I=[\mu'_i : i\in I]$ be a face of $\alpha G^*$; let $\beta_I=\frac{1}{|I|} \sum_{i\in I}\mu_i$ be the barycenter of $G_I$ and $\beta'_I=\frac{1}{|I|} \sum_{i\in I}\mu'_i$ be the barycenter of $G'_I$. Since for all $i\in[d]$, $\mu'_i=a\sum_{j\in[d]\backslash i}\mu_j=-a\mu_i$, it follows that for any proper subset $I$ of $[d]$, $\sum_{i\in I}\mu_i=-\frac{1}{a}\sum_{i\in I}\mu'_i=\frac{1}{a}\sum_{j\in [d]\backslash I}\mu'_j$. Thus, $\beta_I$ is a positive multiple of  $\beta'_{[d]\backslash I}$, and so the ray from $\mathbf{0}$ and through $\beta_{I}$ coincides with the ray from $\mathbf{0}$ and through $\beta'_{[d]\backslash I}$.
Furthermore, since $G$ is regular,  the distance from $\mathbf{0}$ to $\beta_I$ is the same for all $k$-subsets $I$ of $[d]$; we denote it by $\rho_k$ and note that $\rho_1>\dots>\rho_{d-1}$. Similarly, for all $k$-subsets $J$ of $[d]$, the distance from $\mathbf{0}$ to $\beta'_{J}$ is the same number $\rho'_{k}$, where $\rho'_1>\dots>\rho'_{d-1}$. Finally, since $G\subset \inter(\alpha G^*)$, $\rho'_{d-1}>\rho_1$. To summarize,
\begin{equation} \label{eq: rhos} \rho'_1>\dots>\rho'_{d-1}>\rho_1>\dots>\rho_{d-1}.\end{equation}

Consider the $d$-cross-polytope $\CP=\conv(G\times\{1\}  \cup  \alpha G^*\times \{-1\})$. We label the vertices of $\CP$ by $u_j=(\mu_j, 1)$ and $u'_j=(\mu'_j, -1)$ (for $j=1,\ldots,d$), so that $G\times \{1\}=[u_1, \dots, u_d]$ and $\alpha G^*\times \{-1\}=[u'_1, \dots, u'_d]$. For a subset $I$ of $[d]$, we denote the barycenter of $G_I\times \{1\}$ by $b_I$ and the barycenter of and $G'_I\times \{-1\}$ by $b'_{I}$.  Finally, we let $H_I$ denote the hyperplane in $\R^d$ determined by the following set of $d$ points:  $\{u_i : i\in I\}\cup\{u'_j : j\in [d]\backslash I\}$.

\begin{lemma}\label{lm: points of intersection}
	Let $0\leq k\leq d$. Then all hyperplanes $H_I$, where $I\subseteq [d], |I|=k$, intersect the $x_d$-axis at the same point. When $0<k<d$, the $d$th coordinate of this point is $>1$.
\end{lemma}
\begin{proof}
	First note that $H_{[d]}$ and $H_\emptyset$ intersect the $x_d$-axis at $\mathbf{e}_d:=(0, \dots, 0, 1)$ and $-\mathbf{e}_d$, respectively. Now let $I$ be any $k$-subset of $[d]$, where $1\leq k\leq d-1$. Consider the points $b_{I}$ and $b'_{[d]\backslash I}$. Both of them lie in $H_I$; hence, so does the line $\ell=\aff(b_I, b'_{[d]\backslash I})$. 
	
	We claim that $\ell$ intersects the $x_d$-axis. Consequently, $$H_I\cap x_d\mbox{-axis} = \ell \cap x_d\mbox{-axis}.$$ 
	To prove the claim, consider the lines  $\aff(\mathbf{e}_d, b_I)$ and $\aff(-\mathbf{e}_d, b'_{[d]\backslash I})$. By discussion following Definition \ref{def: associated points}, these lines are parallel, and thus determine a $2$-dimensional plane $\mathcal{L}$. For the rest of the proof, we work in this plane. It contains $\ell$ and the $x_d$-axis. Also, since, $\beta_I$ is a positive multiple of $\beta'_{[d]\backslash I}$,  the points $b_I$ and $b'_{[d]\backslash I}$ lie on the same side of the $x_d$-axis in $\mathcal{L}$. Finally, since the distance from $b_I$ to the $x_d$-axis is $\rho_k$, the distance from $b'_{[d]\backslash I}$ to the $x_d$-axis is $\rho'_{d-k}$, and $\rho'_{d-k}>\rho_k$, it follows that $\ell$ and the $x_d$-axis are not parallel. Hence they intersect and the point of intersection, which we denote by $a_I=(0,\dots,0, c_I)$, satisfies $c_I>1$. This proves the claim.
	
	To complete the proof of the lemma, it remains to show that $c_I$ depends only on $|I|=k$. Indeed, consider triangles $[a_I, \mathbf{e}_d, b_I]$ and $[a_I, -\mathbf{e}_d, b'_{[d]\backslash I}]$. They are similar; hence, 
	$$\frac{c_I-1}{\rho_k}=\frac{\dist(a_I, \mathbf{e}_d)}{\dist(\mathbf{e}_d, b_I)}=\frac{\dist(a_I, -\mathbf{e}_d)}{\dist(-\mathbf{e}_d, b'_{[d]\backslash I})}=\frac{c_I+1}{\rho'_{d-k}}.$$
	Solving this equation yields $c_I=\frac{\rho'_{d-k}+\rho_k}{\rho'_{d-k}-\rho_k}$. The result follows.
\end{proof}

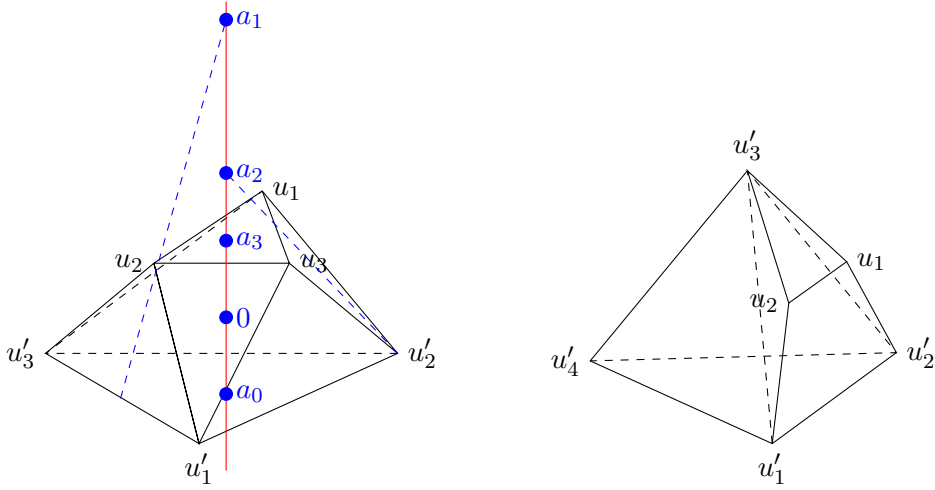
\begin{figure}
	\centering
	\begin{tikzpicture}[scale=1.2]
		\draw (0,0)node[below]{$u'_1$}--(2.2,1)node[right]{$u'_2$}--(1,2)node[right]{$u_3$}--(0,0)--(-0.5, 2)node[left]{$u_2$}--(0,0)--(-1.7, 1)node[left]{$u'_3$}--(-0.5, 2)--(0.7, 2.8)node[right]{$u_1$}--(1,2)--(-0.5, 2);
		\draw (0.7, 2.8)--(2.2,1);
		\draw[dashed] (2.2,1)--(-1.7, 1)--(0.7,2.8);
		\draw[red] (0.3,-0.3)--(0.3, 4.9);
		\draw[blue, dashed] (2.2,1)--(0.3, 3);
		\draw[blue, dashed] (-0.87, 0.5)--(0.3, 4.7);
		\filldraw[blue] (0.3,2.25) circle (2pt) node[right]{$a_3$};
		\filldraw[blue] (0.3,3) circle (2pt) node[right]{$a_2$};
		\filldraw[blue] (0.3,4.7) circle (2pt) node[right]{$a_1$};
		\filldraw[blue] (0.3,1.4) circle (2pt) node[right]{$0$};
		\filldraw[blue] (0.3,0.55) circle (2pt) node[right]{$a_0$};
	\end{tikzpicture}
	\hspace{1cm}
	\begin{tikzpicture}[scale=1.1]
		\draw (0,0)node[below]{$u'_1$}--(0.2, 1.7)node[left]{$u_2$}--(-0.3,3.3)node[above]{$u'_3$}--(0.9, 2.2)node[right]{$u_1$}--(0.2, 1.7);
		\draw (0.9,2.2)--(1.5, 1.1)node[right]{$u'_2$}--(0,0)--(-2.2,1)node[left]{$u'_4$}--(-0.3, 3.3);
		\draw[dashed] (0,0)--(-0.3, 3.3)--(1.5, 1.1)--(-2.2, 1);
	\end{tikzpicture}
	\caption{Left: a pseudo-regular $\CP$ of dimension $3$ and the points $\{a_0, \dots, a_3\}$. Right: The polytope $P^{3,1}$.}
	\label{Fig: pseudo-regular cross-polytope}
\end{figure}

Let $0\leq k\leq d$. In view of Lemma \ref{lm: points of intersection}, we denote by $a_k$ the point of intersection of $H_I$ and the $x_d$-axis, where $I$ is any subset of $[d]$ of size $k$, and by $c_k:=\frac{\rho'_{d-k}+\rho_k}{\rho'_{d-k}-\rho_k}$ the last coordinate of $a_k$; see Figure \ref{Fig: pseudo-regular cross-polytope} for an illustration in dimension $3$. 

\begin{corollary}\label{cor: position of intersections}
    The heights of points $a_1,\ldots,a_d$ satisfy $c_1> \cdots >c_{d-1}> c_d = 1$. In particular, if $q$ is a point on the $x_d$-axis that lies strictly between $a_{k-1}$ and $a_k$, then $q$ is beneath the facet $H_I=[u_i, u'_j: i\in I, j\in [d]\backslash I]$ of the $CP$ if $|I|\leq k-1$, and beyond the facet $H_I$ if $|I|\geq k$.
\end{corollary}
\begin{proof}
By equation \eqref{eq: rhos}, for all $1\leq k\leq d-1$, $\rho'_{d-k}-\rho_k>0$. Hence $c_k=\frac{\rho'_{d-k}+\rho_k}{\rho'_{d-k}-\rho_k}>1=c_d$. Furthermore, for $2\leq k\leq d-1$,\begin{equation*}
		\begin{split}
			& c_k-c_{k-1}=\frac{\rho'_{d-k}+\rho_k}{\rho'_{d-k}-\rho_k}-\frac{\rho'_{d-k+1}+\rho_{k-1}}{\rho'_{d-k+1}-\rho_{k-1}}\\
			=&2\left(\frac{\rho_k}{\rho'_{d-k}-\rho_k} -\frac{\rho_{k-1}}{\rho'_{d-k+1}-\rho_{k-1}}\right)\\
			=&2\left(\frac{1}{\frac{\rho'_{d-k}}{\rho_k}-1}-\frac{1}{\frac{\rho'_{d-k+1}}{\rho_{k-1}}-1} \right)<0, 
		\end{split}
	\end{equation*}
	where the last step follows from the fact that $\rho'_{d-k}>\rho'_{d-k+1}> \rho_{k-1}>\rho_k$; see eq.~\eqref{eq: rhos}.
\end{proof}

\begin{definition} \label{def:a_i}
	Let $\CP=\conv(G\times \{1\} \cup \alpha G^*\times \{-1\})$ be a pseudo-regular $d$-cross-polytope. 
	The set $\{a_k=\cap_{I\subset [d], |I|=k}H_I: 1\leq k\leq d\}$ is called the {\em sequence of points associated with $\CP$}.
\end{definition}

Our construction of a $(d-2)$-simplicial $2$-simple polytope starts with a certain $d$-polytope $P^{d,1}$ described in Definition \ref{def: the first polytope P_1} and proceeds by recursively adding to $P^{d,1}$ a total of $d-3$ additional vertices; see Figure \ref{Fig: pseudo-regular cross-polytope} for an illustration of $P^{3,1}$. As we will see below, one of the facets of $P^{d,1}$ is a pseudo-regular CP (of dimension $d-1$). By a slight abuse of notation, we continue to label the vertices of this facet by $u_1, \dots, u_{d-1}, u'_1, \dots, u'_{d-1}$.

\begin{definition}\label{def: the first polytope P_1}
	Let $\Sigma=[u'_1, ..., u'_{d+1}]$ be a regular $d$-simplex. Choose an arbitrary $0<\epsilon\ll \dist(u'_1, u'_2)$. For $1 \leq i \leq d-1$, let $p_i$ be the barycenter of the $(d-2)$-face $\Sigma\backslash u'_iu'_{d+1}$, and let $u_i := p_i + \epsilon(p_i-u'_{d+1})$. We define $P^{d,1}$ as $\conv(u'_1, \dots, u'_{d+1}, u_1, \dots, u_{d-1})$.
\end{definition}
Since $p_i$ is the barycenter of the $(d-2)$-face $\Sigma\backslash u'_iu'_{d+1}$, it follows that $[p_1, \dots, p_{d-1}]$ is a regular $(d-2)$-simplex and $[p_1, \dots, p_{d-1}, u'_1, \dots, u'_{d-1}]$ is a pseudo-regular $(d-1)$-cross-polytope. By our choice of $u_i$, $[u_1, \dots, u_{d-1}]$ is a regular $(d-2)$-simplex obtained from $[p_1, \dots, p_{d-1}]$ by dilation with factor $(1+\epsilon)$ (where $\epsilon$ is small) followed by translation in the direction perpendicular to  $\aff(p_1,\dots,p_{d-1}, u'_1,\dots,u'_{d-1})=\aff(\Sigma\backslash u'_{d+1})$. In particular, $\aff(u_1, \dots, u_{d-1})$ is parallel to $\aff(u'_1, \dots, u'_{d-1})$ and $\CP:=[u_1, \dots, u_{d-1}, u'_1, \dots, u'_{d-1}]$ is also a pseudo-regular $(d-1)$-cross-polytope.

This discussion shows that the polytope $P^{d,1}$ is the union of the simplex $\Sigma$ and the pyramid with apex $u'_d$ over the cross-polytope $\CP$ (glued along the simplex $[u'_1,\ldots,u'_{d}]$). Furthermore, for each $1\leq i\leq d-1$, the points $\{u_i, u'_1,\ldots,\widehat{u'_i},\ldots, u'_d, u'_{d+1}\}$ lie in the same hyperplane, and, in this hyperplane, the sets $\conv(u_i, u'_{d+1})$ and $\conv(u'_1,\ldots,\widehat{u'_i},\ldots u'_{d})$ intersect in their relative interiors. For $1\leq k\leq d-1$, let $\mathcal{H}_{k}$ be the set of facets $H$ of $\CP$ with $|H\cap \{u_1, \dots, u_{d-1}\}|=k$. (Each such $H$ is a $(d-2)$-face of $P^{d,1}$.) Also, let $H^+ :=H\cap [u_1, \dots, u_{d-1}]$ and $H^-:=H\cap [u'_1, \dots, u'_{d-1}]$. Let $v_0:=u'_{d+1}$ and $v_1:=u'_d$. It follows that $P^{d,1}$ has the following facets:

\begin{enumerate}
	\item The simplex $\Sigma\backslash u'_d$ and the pseudo-regular cross-polytope $\CP$. 
	\item $d-1$ bipyramids of the form $\conv(H\cup \{v_0,v_1\})$, where $H\in\mathcal{H}_1$; the boundary complex of such facet is $\partial(\overline{V(H^+)\cup v_0})*\partial(\overline{V(H^-)\cup v_1})$.
	\item $2^{d-1}-d$ simplex facets of the form $\conv(H \cup v_1)$, where $H\in \cup_{2\leq k\leq d-1} \mathcal{H}_k$.
\end{enumerate}
\noindent In particular, $\CP$ is adjacent to all other facets of $P^{d,1}$. 

 Since $\CP$ is pseudo-regular, by Lemma \ref{lm: points of intersection}, there is a sequence of points associated with $\CP$ (lying in $\aff(\CP)$): $a_i=\cap_{F\in\mathcal{H}_i} \aff(F)$, $1\leq i\leq d-1$; see Definition \ref{def:a_i}. The points $\{a_i:1\leq i\leq d-1\}$ all lie on the line through the barycenters $b_{[d-1]}$ of $[u_1, \dots, u_{d-1}]$ and $b'_{[d-1]}$ of $[u'_1, \dots, u'_{d-1}]$, and, according to Corollary \ref{cor: position of intersections}, they appear on this line in the order $a_1,\ldots,a_{d-2}, a_{d-1}$, with $a_{d-2}$ closest to $a_{d-1}=b_{[d-1]}$ and $a_1$ farthest from $b_{[d-1]}$.

We are now ready for the main definition of this section:

\begin{definition}\label{def: (d-2)-simplicial 2-simple d-polytope}
    Consider the sequence of points $\{a_i: 1\leq i\leq d-2\}$ associated with the facet $\CP=[u_1', \dots, u'_{d-1}, u_1, \dots, u_{d-1}]$ of $P^{d,1}$. Let $v_1=u'_d$. Inductively, for $2\leq i\leq d-2$, choose a point $v_{i}$ in the relative interior of the line segment $[a_i, v_{i-1}]$ and let $P^{d, i}=\conv(P^{d, i-1}\cup v_i)$. Finally, let $P^d=P^{d,d-2}$.
\end{definition}

The process of adding vertices similar to the one described in Definition \ref{def: (d-2)-simplicial 2-simple d-polytope} is illustrated in Figure \ref{fig:adding vertices}, where the vertices are added to the pyramid over a hexagon. (Unfortunately, Definition \ref{def: (d-2)-simplicial 2-simple d-polytope} itself is non-vacuous only when $d\geq 4$, and as such is hard to illustrate.)

Our next goal is to prove that $P^d$ is the promised high-dimensional analog of the $4$-polytope $P_9$; see Theorem \ref{thm: (d-2)-simplicial $2$-simple construction}. This requires describing the facets of $P^d$. We do so by induction, showing that for $2\leq k\leq d-2$,  the set of facets of $P^{d,k}$ is obtained from that of $P^{d,k-1}$  as follows.

\begin{enumerate}
	\item The facets $\conv(H \cup v_{k})$, where $H\in \cup_{k+1\leq i\leq d-1} \mathcal{H}_i$. These facets replace the facets $\conv(H \cup v_{k-1})$ of $P^{d,k-1}$.
	\item $\binom{d-1}{k}$ facets $\conv(H\cup \{v_{k-1},v_k\})$, where $H\in\mathcal{H}_k$. Their boundary complexes are of the form $$\partial(\overline{V(H^+)\cup v_{k-1}})*\partial(\overline{V(H^-)\cup v_{k}}).$$  These facets replace the facets $\conv(H \cup v_{k-1})$ of $P^{d,k-1}$.
	\item The rest of the facets of $P^{d,k-1}$ not mentioned above.
\end{enumerate}

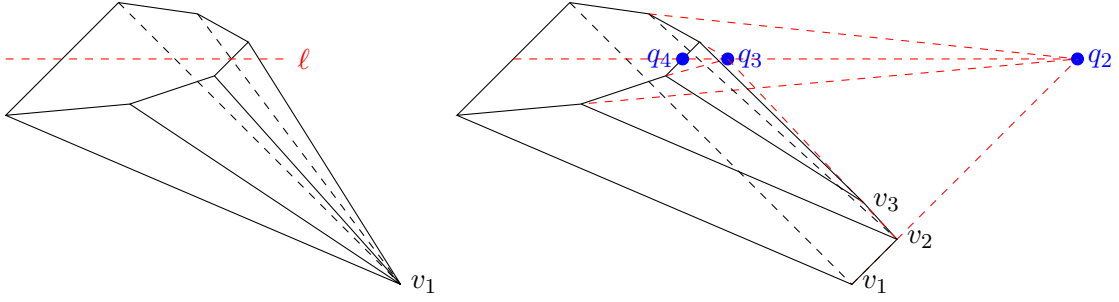
\begin{figure}
	\centering
	\begin{tikzpicture}[scale=1.5]
		\draw (-0.5, -0.5)--(0.6, -0.4)--(1.35,-0.15)--(1.65, 0.15)--(1.2, 0.4)--(0.5, 0.5)--(-0.5, -0.5)--(3, -2)node[right]{$v_1$}--(0.6, -0.4);
		\draw (1.35, -0.15)--(3,-2)--(1.65, 0.15);
		\draw[dashed] (1.2, 0.4)--(3,-2)--(0.5, 0.5);
		\draw[red, dashed] (-0.5, 0)--(2, 0)node[right]{$\ell$};
	\end{tikzpicture}
	\begin{tikzpicture}[scale=1.5]
		\draw (0.5, 0.5)--(-0.5, -0.5)--(0.6, -0.4)--(1.35,-0.15)--(1.65, 0.15)--(1.2, 0.4)--(0.5, 0.5);
		\draw[dashed, red] (0,0)--(5, 0)--(0.6, -0.4);
		\draw[dashed, red] (1.35, -0.15)--(1.9, 0)--(1.65, 0.15);
		\draw[dashed, red] (1.2, 0.4)--(5, 0);
		\filldraw[blue] (1.5,0) circle (1.5pt) node[left]{$q_4$};
		\filldraw[blue] (1.9,0) circle (1.5pt) node[right]{$q_3$};
		\filldraw[blue] (5,0) circle (1.5pt) node[right]{$q_2$};
		\draw[dashed, red] (3, -2)--(5,0);
		\draw (0.6, -0.4)--(3.4, -1.6)node[right]{$v_2$}--(3, -2)node[right]{$v_1$}--(-0.5, -0.5);
		\draw (1.35, -0.15)--(3.1, -1.27)node[right]{$v_3$}--(3.4, -1.6);
		\draw [dashed, red] (1.9, 0)--(3.4, -1.6);
		\draw (1.65, 0.15)--(3.1, -1.27);
		\draw [dashed] (0.5, 0.5)--(3, -2);
		\draw[dashed] (1.2, 0.4)--(3.4, -1.6);
	\end{tikzpicture}
	\caption{Left: The pyramid over a hexagon $H$ symmetric about the line $\ell$. Right: A new $3$-polytope obtained by adding vertices $v_2$ and $v_3$, with $v_{i+1}$ in the interior of the line segment $[q_{i+1}, v_i]$; here $q_{i+1}$ is the intersection of affine spans of the appropriate symmetric edges of $H$.}
	\label{fig:adding vertices}
\end{figure}

\noindent In particular, it follows by induction that $\CP$ is a facet of $P^{d,k}$ and that it is adjacent to {\em all} other facets of $P^{d,k}$, and, furthermore, that the collection of facets in item 3 consists of $\Sigma\backslash u'_d$, $\CP$, and for each $1\leq i\leq k-1$ and  $H\in \mathcal{H}_i$, a facet that contains $H\cup v_i$.

The proof is based on:

{\bf Claim 1}: {\em For every $H\in \mathcal{H}_k$, $v_k\in \aff(H \cup v_{k-1})$.} This is because $a_k$ lies on the hyperplane $\aff(H)$, and $v_k\in [a_k, v_{k-1}]$. 

{\bf Claim 2}: {\em For $i>k$ and $H\in \mathcal{H}_i$, $v_k$ is beyond $\conv(H \cup v_{k-1})$.}
Indeed, by Corollary \ref{cor: position of intersections}, in $\aff(\CP)$, $a_k$ is beyond $H$. Hence in $\aff(\CP \cup v_{k-1})=\mathbb{R}^d$, the point $v_k\in \inter [a_k, v_{k-1}]$ is beyond $\conv(H\cup v_{k-1})$.

{\bf Claim 3}: {\em $v_k$ is beneath the rest of the facets of $P^{d,k-1}$.} First, as easily seen from the definition of sequences $\{a_j\}$ and $\{v_j\}$, $v_k$ is beneath both $\Sigma\backslash u'_d$ and $\CP$. Thus it only remains to show that if $G$ is a facet of $P^{d,k-1}$ that contains $H\cup v_i$ for some $i<k$ and $H\in \mathcal{H}_i$, then $v_k$ is beneath $G$. This follows from Corollary \ref{cor: position of intersections} along with another simple induction on $j$, where $i+1\leq j\leq k$. For the base case, by Corollary \ref{cor: position of intersections}, in $\aff(\CP)$, $a_{i+1}$ is beneath $H$. Hence, in $\aff(\CP \cup v_{i})=\R^d$, $a_{i+1}$ is beneath $G$. Since $v_{i+1}$ is in the interior of $[v_i, a_{i+1}]$, $v_{i+1}$ is also beneath $G$. The inductive step is very similar: by the inductive hypothesis, $v_j$ is beneath $G$ and by Corollary \ref{cor: position of intersections}, so is $a_{j+1}$; hence $v_{j+1}\in[v_j,a_{j+1}]$ is also beneath $G$. The claim follows.

The above three claims uniquely determine the facets of $P^{d, k}$. Claim 3 implies that the facets of $P^{d,k-1}$ from item 3 in the list are unaffected by adding $v_k$, and hence remain facets of $P^{d,k}$. 

Claim 1 implies that for every $H\in \mathcal{H}_k$, the facet $\conv(H\cup v_{k-1})$ of $P^{d, k-1}$ is replaced by a new facet $\conv(H\cup \{v_k, v_{k-1}\})$. Note that the barycenter $b_{H^+}$ of $H^+$ lies on the line segment connecting $a_k$ and the barycenter $b_{H^-}$ of $H^-$ (see the proof of Lemma \ref{lm: points of intersection}). Hence, if $v_k$ is an interior point of the line segment $[a_k, v_{k-1}]$, then $[b_{H^+}, v_{k-1}]$ and $[b_{H^-}, v_{k}]$ intersect at a point $p$. This implies that $\conv(H^+ \cup v_{k-1})\cap \conv(H^- \cup v_{k})=p$. Thus the boundary complex of $\conv(H \cup \{v_k, v_{k-1}\})$ must be $\partial(\overline{V(H^+)\cup v_{k-1}})*\partial (\overline{V(H^-)\cup v_{k}})$. These facets are exactly\footnote{To see this, we invite the reader to compute the link of $v_{k-1}v_k$ in the polytopal complex generated by these facets and check that it is a $(d-3)$-dimensional pseudomanifold (i.e., every ridge is in two facets). Thus it must coincide with the link of $v_{k-1}v_k$  in the boundary of $P^{d,k}$.} the facets of $P^{d, k}$ containing $v_{k-1}v_k$.

Finally, the rest of the facets of $P^{d, k}$ are those arising from $H\in\mathcal{H}_i$ for $i>k$. By Claim 2 and the previous paragraph, they must be of the form $\conv(H\cup v_k)$, replacing $\conv(H\cup v_{k-1})$ of $P^{d, k-1}$.

We thus obtain the following result (for convenience we let $v_{d-1}=v_{d-2}$):
\begin{lemma}\label{lm: list of facets}
	The polytope $P^d$ in Definition \ref{def: (d-2)-simplicial 2-simple d-polytope} has $3(d-1)$ vertices and $2^{d-1}+1$ facets. The vertex set of $P^d$ is $$\{u_1, \dots, u_{d-1}, u'_1, \dots, u'_{d-1}, u'_d=v_1, u'_{d+1}=v_0, v_2, \dots, v_{d-3}, v_{d-2}=v_{d-1}\}.$$ The set of facets of $P^d$ naturally splits into the following $d$ subfamilies:
	\begin{enumerate}
		\item $\mathcal{F}_0$ consists of  the simplex $[u'_1, \dots, u'_{d-1}, u'_{d+1}]$ and the cross-polytope $\CP$.
		\item For $1\leq k\leq d-1$, $\mathcal{F}_k$ consists of $\binom{d-1}{k}$ polytopes of dimension $d-1$ whose boundary complexes are of the form $\partial(\overline{V(H^+)\cup v_{k-1}})*\partial (\overline{V(H^-)\cup v_{k}})$, where $H\in \mathcal{H}_{k}$. In particular, $\mathcal{F}_{d-1}=\{[u_1, \dots, u_{d-1}, v_{d-2}]\}$.
	\end{enumerate}
\end{lemma}
\begin{theorem}\label{thm: (d-2)-simplicial $2$-simple construction}
	The $d$-polytope $P^d$ is $(d-2)$-simplicial and $2$-simple. It has two pairs of a simplex facet and a simple vertex not in that facet; they are  $([u_1,\ldots, u_{d-1},v_{d-2}], u'_{d+1})$ and $([u'_1, \dots, u'_{d-1}, u'_{d+1}], v_{d-2})$.
\end{theorem}
\begin{proof}
Let $U=\{u_1, \dots, u_{d-1}\}$ and let $U'=\{u'_1, \dots, u'_{d-1}\}$. For $M=\{u_{i_1},\ldots,u_{i_k}\}\subseteq U$, we let $M':=\{u'_{i_1},\ldots,u'_{i_k}\}\subseteq U'$. Also, for brevity, we write $u$, $uv$, $uvw$ instead of $\{u\}$, $\{u,v\}$, and $\{u,v,w\}$.

The description of facets in Lemma \ref{lm: list of facets} guarantees that $P^d$ is $(d-2)$-simplicial. To show that $P^d$ is also $2$-simple, it suffices to check that every $(d-3)$-face $\tau$ of $P^d$ is contained in exactly three facets. 
By examining families $\mathcal{F}_i$, $0\leq i \leq d-1$, of Lemma \ref{lm: list of facets}, we see that there are the following possible cases:
	\begin{enumerate}
		\item $u'_{d+1}\in V(\tau)$. In this case, $V(\tau)\subset U'\cup u'_d u'_{d+1}$. If $u'_d$ is also in $\tau$, then $\tau$ is contained in three bipyramids from $\mathcal{F}_1$; otherwise, $\tau$ is contained in two bipyramids from $\mathcal{F}_1$ and the simplex $[u'_1, \ldots,u'_{d-1},u'_{d+1}]$ from $\mathcal{F}_0$.
		\item $V(\tau) \subset U'$. In this case, $\tau$ is contained in the cross-polytope and the simplex from $\mathcal{F}_0$, and one bipyramid from $\mathcal{F}_1$.
		\item $V(\tau)=K\cup M'$, where $K\sqcup M \sqcup u_\ell=U$ and $|K|=i$ for some $1\leq \ell\leq d-1$ and $1\leq i\leq d-2$. Then $\tau$ is a face of $\CP$ from $\mathcal{F}_0$, of $\partial(\overline{K\cup u_\ell v_{i}})*\partial(\overline{M'\cup v_{i+1}})$ from $\mathcal{F}_{i+1}$, and of $\partial(\overline{K\cup v_{i-1}})*\partial(\overline{M' \cup u'_\ell v_{i}})$ from $\mathcal{F}_{i}$.
		\item $V(\tau)=K\cup M'\cup v_i$, where $1\leq i\leq d-2$ and $K\sqcup M\sqcup u_ju_k=U$ for some $1\leq j<k\leq d-1$. There are two cases:
		\begin{enumerate}
			\item $|K|=i-1$. Then $\tau$ is a face of $\partial(\overline{K\cup u_ju_kv_{i}})*\partial(\overline{M'\cup v_{i+1}})$ from $\mathcal{F}_{i+1}$ and of two facets $\partial(\overline{K\cup u_jv_{i-1}})*\partial(\overline{M'\cup u'_kv_{i}})$, $\partial(\overline{K\cup u_kv_{i-1}})*\partial(\overline{M'\cup u'_jv_{i}})$ from $\mathcal{F}_{i}$.
			\item $|K|=i$ (and so, $i<d-2$). Then $\tau$ is a face of $\partial(\overline{K\cup v_{i-1}})*\partial(\overline{M'\cup u'_ju'_kv_{i}})$ from $\mathcal{F}_{i}$. and of two facets $\partial(\overline{K\cup u_jv_{i}})*\partial(\overline{M'\cup u'_kv_{i+1}})$, $\partial(\overline{K\cup u_kv_{i}})*\partial(\overline{M'\cup u'_jv_{i+1}})$ from $\mathcal{F}_{i+1}$.
		\end{enumerate}
		\item $V(\tau)=K\cup M'\cup v_{i-1}v_i$, where $2\leq i\leq d-2$ and $K\sqcup M\sqcup u_ju_ku_\ell=U$ for some $1\leq j<k<\ell\leq d-1$. There are two cases:
		\begin{enumerate}
			\item $|K|=i-2$. Then $\tau$ is contained in three facets from $\mathcal{F}_{i}$: 
			\begin{eqnarray*}\partial(\overline{K\cup u_ku_\ell v_{i-1}})*\partial(\overline{M'\cup u'_jv_{i}}), \quad \partial(\overline{K\cup u_ju_\ell v_{i-1}})*\partial(\overline{M'\cup u'_kv_{i}}), \mbox{ and }\\
			\partial(\overline{K\cup u_ju_k v_{i-1}})*\partial(\overline{M'\cup u'_\ell v_{i}}).
			\end{eqnarray*}
			\item $|K|=i-1$. Then $\tau$ is contained in three facets from $\mathcal{F}_i$: 
			\begin{eqnarray*}\partial(\overline{K\cup u_\ell v_{i-1}})*\partial(\overline{M'\cup u'_ju'_kv_{i}}), \quad \partial(\overline{K\cup u_j v_{i-1}})*\partial(\overline{M'\cup u'_ku'_\ell v_{i}}), \mbox{ and} \\
			\partial(\overline{K\cup u_k v_{i-1}})*\partial(\overline{M'\cup u'_ju'_\ell v_{i}}).\end{eqnarray*}
		\end{enumerate}
	\end{enumerate} 
The result follows.
\end{proof}

\begin{remark}It is worth noting that the polytope $P^d$ is $d$-dimensional and has $3d-3$ vertices. This is the smallest number of vertices that a non-simplex $(d-2)$-simplicial $2$-simple $d$-polytope can have in dimensions $d=3,4,5$ (cf. Proposition \ref{vertex number:d=5,i=2}).
\end{remark}

As the last theorem of the paper, we show that iteratively merging $n$ copies of $P^d$ from Theorem \ref{thm: (d-2)-simplicial $2$-simple construction} results in exponentially many (w.r.t.~the number of vertices) combinatorially distinct $(d-2)$-simplicial $2$-simple $d$-polytopes. Recall from Theorem \ref{thm: (d-2)-simplicial $2$-simple construction} that 
	\begin{itemize}
		\item The polytope $P^{d}$ has two simple vertices $u'_{d+1}$ and $v_{d-2}$, and two simplex facets $F':=[u'_1, \dots, u'_{d-1}, u'_{d+1}]$ and $F:=[u_1,\dots, u_{d-1}, v_{d-2}]$; $u'_{d+1}$ is a vertex of $F'$ but not of $F$, and $v_{d-2}$ is a vertex of $F$ but not of $F'$.  All other facets containing $u'_{d+1}$ and $v_{d-2}$ are bipyramids.
		\item The $\CP$ facet $[u_1, \dots, u_{d-1}, u'_1, \dots, u'_{d-1}]$ is adjacent to all other facets of $P^d$. 
	\end{itemize} 
	
Let $T_1$ and $T_2$ be two copies of $P^d$ with the copy of  $\CP$, $F$, and $F'$ in $T_i$ denoted by $\CP_i$, $F_i$, and $F'_i$, respectively, and the copy of $u'_{d+1}$ from $T_2$ denoted by $w$. We merge $T_1$ and $T_2$ along $F_1$ and $w$. Since $\CP_1$ is adjacent to $F_1$, and since $w$ is in one simplex facet (namely $F'_2$) and $d-1$ bipyramids, exactly as in the $4$-dimensional case, there are two ways to merge leading to two distinct combinatorial types (recall that $\sigma_{d-1}$ denotes a $(d-1)$-simplex):

\begin{itemize}
	\item 
	In $T_1\merge T_2$, the facet $\CP_1$ gets merged with the simplex $F'_2$. The merged facet is then again a $\CP$. Since $\CP_2$ is adjacent to all other facets of $T_2$, including $F'_2$, it follows that the polytope $T_1\merge T_2$ has two $\CP$ facets and that they are adjacent to each other.
	\item 
	In $T_1\merge T_2$, the facet $\CP_1$ gets merged with a bipyramid, resulting in a facet of the form $\CP\#\sigma_{d-1}$. In this case, $T_1\merge T_2$ has two ``large'' facets: $\CP\#\sigma_{d-1}$ and $\CP_2$, and they are adjacent to each other; every other facet has at most $d+1$ vertices.
\end{itemize}
With these observations in hand, we are ready to prove the following.
\begin{theorem}\label{prop: counting combinatorial types}
	There are $2^{\Omega(N)}=2^{\Omega(k)}$ combinatorially distinct $(d-2)$-simplicial $2$-simple $d$-polytopes with $N=(3d-3)+k(2d-4)$ vertices.
\end{theorem}
\begin{proof}
	 Consider $k+1$ copies of $P^d$, which we denote by $T_1,\dots,T_{k+1}$, with the corresponding copies of the CP facet denoted by $\CP_i$. Each $T_i$ has two pairs of a simplex facet and a simple vertex not in that facet, which in this proof we will denote by $(F_i, w_i)$ and $(F'_i, w'_i)$. Consider all polytopes resulting from $(\cdots((T_1\merge T_2)\merge T_3)\cdots)\merge T_{k+1}$ by the following rules:
	
	\begin{itemize}
		\item In the first step, we merge $T_1$ and $T_2$ so that the facet $\CP_1$ is merged with a bipyramid. In step $i$ where $2\leq i\leq k$, we have two choices of whether we merge $\CP_i$ with a simplex or with a bipyramid. 
		\item In the $i$th step, when computing the merge of $(\cdots((T_1\merge T_2)\merge T_3)\cdots)\merge T_{i}$ with $T_{i+1}$, we always merge along $F_i$ and $w_{i+1}$. 
	\end{itemize}
Denote by $R_k$ the polytope obtained in the $k$th step. In the $i$th step ($1\leq i<k$), $F_{i+1}$ from $T_{i+1}$ remains untouched and can be used for the $(i+1)$st step. For $1\leq j\leq k+1$, we refer to the facet of $R_k$ resulting from $\CP_j$ as the $j$th {\em special facet}. By remarks above, for each $1\leq j\leq k$, the $j$th special facet is either a $\CP$ or a $\CP\#\sigma_{d-1}$; the $(k+1)$st special facet is always a $\CP$. Furthermore, for all $1\leq i,j \leq k+1$, the $i$th and $j$th special facets are adjacent if and only if $|i-j|=1$.
	
We show that this procedure produces at least $2^{k-1}$ pairwise non-isomorphic polytopes. First note that the boundary complexes of all non-special facets of $R_k$ are either simplices, joins of two simplices, or stackings over these, and so a non-special facet can never be isomorphic to $\CP$ or $\CP\#\sigma_{d-1}$.  Associate with $R_k$ its {\em profile} which is given by the following abstract graph: the nodes represent the facets of the form $\CP$ and $\CP\#\sigma_{d-1}$, 
and two such nodes are connected by an edge if the corresponding facets are adjacent; also, label each node with a $0$ or $1$ depending on whether it represents a facet that is a CP or a $\CP\#\sigma_{d-1}$. The resulting profile is then a {\em path} with $k+1$ nodes labeled by $0$'s and $1$'s; one of the endpoints is always labeled by $1$ (this is ensured by the first bullet point) and the other endpoint is always labeled by $0$. The one labeled by $1$ corresponds to $T_1$ (the left-most copy of $P^d$)  and the one by $0$ corresponds to $T_{k+1}$ (the right-most copy of $P^d$). 

There are $2^{k-1}$ such $0/1$-paths, and we claim that each of them is a valid profile. Indeed, given such a path, walk along it from the endpoint labeled by $1$ to the endpoint labeled $0$ and read the labels of the nodes. The node at distance $i-1$ from the first endpoint corresponds to the special facet coming from $T_i$ and the label of that node simply tells us whether at the $i$th step we should merge $\CP_i$ with a simplex of with a bipyramid. This claim completes the proof since isomorphic polytopes have the same profile. In other words, two polytopes with distinct profiles have different combinatorial types.
\end{proof}
\begin{remark} \label{rem:exp-many-d=4}
	When $d=4$, we can further merge $R_k$ with a $2$-simplicial $2$-simple $4$-polytope with $10$, $11$, or $16$ vertices. Such polytopes can be found in \cite[Section 4.1]{PafWer}, where they are denoted by $P_{10}, P_{11}, P_{16}=\mathcal{I}^1(P_{11})$. 
	This allows us to create exponentially many (in $N$) $2$-simplicial $2$-simple $4$-polytopes with $N$ vertices for all sufficiently large integers $N$ (not just those with $N\equiv 1 \mod 4$). It follows from Corollary \ref{prop: f-vector} that all resulting polytopes are elementary. Hence for $d=4$, the number of combinatorially distinct $2$-simplicial $2$-simple $4$-polytopes  that are also elementary grows exponentially with the number of vertices. This strengthens \cite[Corollary 4.2]{PafZieg}.
\end{remark}

\subsection*{Acknowledgments} We are grateful to Amzi Jeffs, Josh Hinman, and G\"unter Ziegler for their comments on the preliminary version of this paper, and to Marge Bayer and Eran Nevo for insightful conversations.

{\small
	\bibliography{refs}
	\bibliographystyle{plain}
}

\end{document}